\documentclass[reqno,12pt,letterpaper]{amsart}

\RequirePackage{amsmath,amssymb,amsthm,graphicx,mathrsfs,url}
\RequirePackage[usenames,dvipsnames]{color}
\RequirePackage[colorlinks=true,linkcolor=ForestGreen,citecolor=MidnightBlue]{hyperref}
\RequirePackage{amsxtra}
\usepackage{comment}
\usepackage{enumitem}
\usepackage{cases}

\relax

\numberwithin{equation}{section}

\newtheorem{theo}{Theorem}
\newtheorem{prop}{Proposition}[section]
\newtheorem{defi}[prop]{Definition}
\newtheorem{lemm}[prop]{Lemma}

\def\Remark{\noindent\textbf{Remark.}\ }
\def\Remarks{\noindent\textbf{Remarks.}\ }

\newtheorem{defn}[prop]{Definition}
\newtheorem{ass}[prop]{Assumption}

\newtheorem{rem}[prop]{Remark}

\usepackage{url}

\newcommand{\eps}{{\varepsilon}}

\newcommand{\RR}{{\mathbb R}}

\newcommand\cO{{\mathcal O}}

\newcommand{\ox}{\overline{x}}
\newcommand{\oy}{\overline{y}}
\newcommand{\otheta}{\overline{\theta}}
\newcommand{\oxi}{\overline{\xi}}
\newcommand{\oalpha}{\overline{\alpha}}
\newcommand{\ob}{\overline{b}}

\newcommand{\urho}{\underline{\rho}}

\newcommand{\usigma}{\underline{\sigma}}

\def\eps{\epsilon }

\newcommand\adots{\mathinner{\mkern2mu\raise1pt\hbox{.}
\mkern3mu\raise4pt\hbox{.}\mkern1mu\raise7pt\hbox{.}}}

\let\oldtocsection=\tocsection
\let\oldtocsubsection=\tocsubsection
\renewcommand{\tocsection}[2]{\hspace{0em}\oldtocsection{#1}{#2}}
\renewcommand{\tocsubsection}[2]{\hspace{1em}\oldtocsubsection{#1}{#2}}

\numberwithin{equation}{section}

\begin{document}

\title[Convex grazing waves ]{Convex waves grazing convex obstacles to high order}

\author{Jian Wang}
\email{wangjian@ihes.fr}
\address{Institut des Hautes {\'E}tudes Scientifiques, Bures-sur-Yvette, France 91893}
\author{Mark Williams}
\email{williams@email.unc.edu}
\address{Department of Mathematics, University of North Carolina, Chapel Hill, NC 27599}

\begin{abstract}
In a recent paper \cite{ww2023} we studied the transport of oscillations in solutions to
linear and some semilinear second-order hyperbolic boundary problems along rays that graze a convex obstacle to any order.  We showed that high frequency exact solutions are well approximated  in $H^1$ by much simpler approximate solutions constructed from explicit solutions to profile equations.   That result depends on two geometric assumptions, referred to here as the grazing set (GS) and reflected flow map (RFM) assumptions, that are both difficult to verify in general.    The GS assumption states that the \emph{grazing set}, that is, the set of points on the spacetime boundary at which incoming characteristics meet the boundary tangentially, is a codimension two, $C^1$ submanifold of spacetime.  The second is that the \emph{reflected flow map}, which sends points on the spacetime boundary forward in time to points on reflected and grazing rays,  is injective and has appropriate regularity properties near the grazing set.   In this paper we analyze these assumptions for incoming plane, spherical, and more general ``convex waves" when the governing hyperbolic operator is the wave operator $\Box:=\Delta-\partial_t^2$.   We prove general results describing when the assumptions hold, and provide explicit examples where the GS assumption fails.

\end{abstract}

\maketitle

{
\hypersetup{linkcolor=NavyBlue}
\tableofcontents
}

\newpage

\section{Introduction}

In the paper \cite{ww2023} we provided a geometric optics description in spaces of low regularity of the transport of  oscillations in solutions to linear and some semilinear second-order hyperbolic boundary problems
along rays that graze the boundary of a convex obstacle to arbitrarily high finite or infinite order.    The fundamental motivating example  is the case where the spacetime manifold is $M=(\mathbb{R}^n\setminus \mathcal{O})\times \mathbb{R}_t$, 
where $\mathcal{O}\subset \mathbb{R}^n$ is an open convex  
obstacle with $C^\infty$ boundary, and the governing hyperbolic operator is the wave operator $\Box:=\Delta-\partial_t^2$.

The main theorem of \cite{ww2023}, which applies to quite general second-order hyperbolic operators, obstacles, and incoming waves,  showed that high frequency exact solutions with high order grazing rays are well approximated  in spaces of low regularity by approximate solutions constructed from fairly explicit solutions to simple profile equations.  The theorem has two main assumptions, both of a geometric nature, which need to be verified for each choice of hyperbolic operator, obstacle, and incoming wave.  The first is that the \emph{grazing set}, that is, the set of points on the spacetime boundary at which incoming characteristics meet the boundary tangentially, is a codimension two, $C^1$ submanifold of spacetime.  The second is that the \emph{reflected flow map}, which sends points on the spacetime boundary forward in time to points on reflected and grazing rays,  is injective and has appropriate regularity properties near the grazing set.   Both assumptions are in general hard to verify.  In this paper we  formulate general results describing when the assumptions hold, and provide explicit examples where the GS assumption fails.

In \S \ref{dar} we explain how to distinguish different \emph{orders} of grazing.  The orders may be $2,4, 6,\dots$ or $\infty.$
We showed in \cite{ww2023} that when the order is $2$, the reflected flow map (RFM) and grazing set (GS) assumptions always hold for very general second order hyperbolic operators, obstacles,  and incoming waves.   When the order of grazing is higher than two,  the GS assumption can fail even for $\Box$, strictly convex obstacles, and incoming \emph{plane} waves as we show in $\S\ref{bad}$, and the RFM assumption becomes much harder to verify even in such cases.  In \cite{ww2023} we identified large classes of strictly convex obstacles in all dimensions for which both assumptions are satisfied when incoming \emph{plane waves} for $\Box$ graze the obstacle to arbitrarily high finite or infinite order. 

Plane waves may be thought of as being produced by a point source at infinity.   Much of this paper is concerned with verifying the RFM  and GS assumptions for incoming \emph{spherical waves} for $\Box$, which may be viewed as arising from a point source  at a finite distance from the obstacle.   This is done in \S \ref{siw} and \S\ref{natureg}--\ref{3do}.   
The incoming waves we consider in this paper are associated to phase functions of the form
\[
 \phi_i(t,x) = -t+\psi_i(x), 
 \]
where $\psi_i:\Omega\to \mathbb{R}$, $\Omega\subset \RR^n$, is a smooth convex function that satisfies the eikonal equation
\[(\partial_{x_1}\psi_i)^2+|\nabla \psi_i|^2=1, \ \nabla:=(\partial_{x_2},\cdots, \partial_{x_n}).\]
These are the \emph{convex waves} referred to in the title.   When $\psi_i(x)=\theta\cdot x$ for $\theta\in S^{n-1}$ we obtain plane waves; when $\psi_i(x)=|x-b|$ we obtain spherical waves associated to a source at $b\in\mathbb{R}^n$.  \S \ref{ciw} verifies the RFM assumption for general convex incoming waves.  Many examples of such waves are provided in the Remark after Assumption \ref{assp}.

   While the RFM assumption turns out to hold for general convex waves and strictly convex obstacles, this is not the case for the  GS assumption, where the dimension of the obstacle is now important.    
In \S \ref{natureg} we show that for  2D strictly convex obstacles the GS assumption always holds for incoming spherical waves.\footnote{In \cite{ww2023} the GS assumption was shown to hold always for incoming planar waves and 2D strictly convex obstacles.}    \S \ref{nd} verifies the GS assumption for nD obstacles with a special symmetry; these include examples of infinite order tangency.   In \S\ref{3do} we  
verify the GS assumption for incoming spherical waves for a large class of 3D strictly convex obstacles.   Theorem \ref{u1} gives a simple sufficient condition \eqref{u1ww} on the lowest order (nonconstant) homogeneous polynomial in the Taylor expansion of the function defining the obstacle for the grazing set to be a codimension two $C^\infty$ submanifold of spacetime.  

In \S\ref{bad} we provide  
examples, both for incoming plane and spherical waves,   where the GS assumption \emph{fails}; in these examples the grazing set is \emph{cusped};  it is the graph of a function that is continuous, but not $C^1$.  We also give examples  where the grazing set is $C^1$ but not twice differentiable, so the GS assumption holds, but just barely.  These examples 
show what can happen when the obstacle does not satisfy the sufficient condition \eqref{u1ww} of Theorem \ref{u1}. 
 
 Theorem \ref{A} of \S \ref{branch} provides a general result about the nature of the grazing set for strictly convex obstacles for which condition 
\eqref{u1ww} of Theorem \ref{u1} need not hold, but where the high order grazing point is an isolated zero of the Gauss curvature of the boundary of the obstacle.  The result is that, while the grazing set may not be $C^1$, it is at least $C^0$.  More precisely, the spatial projection of the grazing set is a \emph{single} continuous curve near the high order grazing point; branching does not occur.

It is perhaps surprising that nonsmooth grazing sets can occur even in problems where the convex obstacle has a $C^\infty$ boundary and the incoming waves are plane waves or spherical waves.   The examples show that for  strictly convex obstacles that do not satisfy condition \eqref{u1ww} of Theorem \ref{u1},  the GS assumption can fail or not depending on the location of the source $b$ of spherical waves with phase $-t+|x-b|$.

   The spatial projection of the grazing set for incoming spherical waves can be defined as the intersection of two smooth hypersurfaces in $\mathbb{R}^n$.  
   The regularity of that set  is difficult to study in general, 
   because the normals to these two hypersurfaces always point in the \emph{same direction} at a point of high order (that is, order $\geq 4$ in the sense of Definition \ref{d13}) grazing.   Thus, the implicit function theorem \emph{cannot} be applied directly; see Remark \ref{d14z}.  In \S \ref{gsa} we introduce some methods for dealing with this failure of transversality.
   
   We believe that the study of the reflected flow map and the grazing set are of independent interest, apart from the connection to the main theorem of \cite{ww2023}.   So we have tried to write this paper so that it can be read independently of \cite{ww2023}. 
Sections \ref{siw}, \ref{gsa}, \ref{ciw} use notation and definitions from section \ref{dar}, but can be read independently of each other.
   \vspace{5pt}

\noindent
\textbf{Remark on notations.} 

\noindent
1. For the dot product of vectors $u,v\in\mathbb{R}^k$ we use both $u\cdot v$ and $\langle u,v\rangle$.

\noindent
2.  If $x,y$ are row vectors in $\mathbb{R}^n$, we sometimes denote by $x\otimes y$ the $n\times n$ matrix $x^Ty$.

\noindent
3. We denote $x=(x_1,\ox)\in \RR^{1+(n-1)}$ and set $\nabla = \nabla_{\ox} = (\partial_{x_2}, \cdots, \partial_{x_n})$.
Otherwise, when we use $\nabla$ we will attach a subscript variable as in Definition \ref{def: convex}.

\noindent
4.   If $S\subset \mathbb{R}^n_x\times \mathbb{R}_t$, we denote by $\pi_xS$ the projection of $S$ onto $\mathbb{R}^n_x$, by $\pi_{\ox}S$ the projection onto $\mathbb{R}^{n-1}_{\ox}$, and so on.

   \section{Definitions, results, and applications}\label{dar}
   In this section we explain the RFM and GS assumptions.   We begin by clarifying  the definition of $C^1$ convex (concave) functions and strictly convex (concave) functions.

   \begin{defi}\label{def: convex}
    Let $U\subset\mathbb{R}^n$ be an open convex set.\footnote{This means that if $x,y\in U$, then the line segment $[x,y]$ is contained in $U$.}
    \begin{enumerate}
        \item A $C^1$ function $f: U\to \RR$ is said to be convex (or convex on $U$) if for any $z_1,z_2\in U$, 
        \begin{equation}\label{convex}
        f(z_2)-f(z_1)\geq \langle \nabla_z f(z_1), z_2-z_1 \rangle;
        \end{equation}
        here $\nabla_z f=(\partial_{x_1}f, \cdots, \partial_{x_{n}} f)$. A $C^1$ function is said to be concave if $\geq$ in \eqref{convex} is replaced by $\leq$.
        \item A $C^1$ function $f: U\to \RR$ is said to be strictly convex/convave, if \eqref{convex} (resp. \eqref{convex} with $\leq$) holds for all $z_1, z_2\in U$ with equality holding if and only if $z_1=z_2$.
    \end{enumerate}
\end{defi}

\noindent \textbf{Remarks and Examples. }

\noindent
1. This definition appears to be standard. For example, it is the one given in \cite{robertsvarberg1973}.
For $f:U\to \mathbb{R}$ as above and any subset $V\subset U$, we say that $f$ is convex/concave on $V$ if it is convex/concave on $U$.

\noindent
2. A $C^2$ function on $U$  is convex/concave if and only if its Hessian matrix is positive/negative semi-definite.

\noindent
3. If a $C^2$ function has a positive/negative definite Hessian matrix, then it is strictly convex/concave.  The converse is not true.

\noindent
4. (a) For $k\in \mathbb{N}$ the function $f(z)=|z|^{2k}$ is strictly convex. When $k>1$  we have $\nabla_z^2f>0 \Leftrightarrow z\neq 0$.

(b) For $k\geq 1$ the function $f(z)=z_1^{2k}+\cdots+z_n^{2k}$ is strictly convex.   We have $\nabla_z^2f\geq 0$,  
but for $k>1$ the function fails to satisfy the condition that $z\neq 0 \Rightarrow \nabla_z^2f>0$.

\noindent 
5. The function, $f(z)=\begin{cases}e^{-\frac{1}{|z|}},\;z\neq 0\\ 0, \;z=0\end{cases}$ is $C^\infty$, strictly convex near $z=0$, and satisfies $\nabla_z^2 f>0\Leftrightarrow z\neq 0$ near $z=0$. The function, $g(z)=|z|^{a}$, $a>1$, $a\in \RR$, is $C^1$ and strictly convex.


\begin{ass}\label{n1z}
Let  $\mathcal O\subset \RR^n$  be a convex obstacle given by 
\begin{align}\label{defo}
\mathcal O:=\{ (x_1, \ox)| \ x_1< F(\ox),\; \ox\in\RR^{n-1} \},
\end{align}
where $F: \RR^{n-1}\to \RR$ is a $C^\infty$ strictly concave function, and set $M:=(\mathbb{R}^n\setminus \mathcal{O})\times \mathbb{R}_t$.   We arrange by translation and rotation  of $\mathcal{O}$ so that 
 $F$ has a max at $\ox=0$ with $F(0)=1$.   In that case $\nabla F(0)=0$ and the tangent plane to $\partial\mathcal{O}$ at $x=(x_1,\ox)=(1,0)$ is   $x_1=1$. 
\end{ass}




\subsection{Incoming and reflected phases}
The ``waves" that we consider are solutions to a continuation/reflection problem on $M_T:=(\mathbb{R}^n\setminus \mathcal{O})\times [-T,T]$ of the form\footnote{The exact meaning of ``$\sim$" in \eqref{e0ac} is not important at this point; see Remarks (1) at the end of section \ref{dar}.}
\begin{subnumcases}{\label{e0a}}
    \Box u^\eps=f(x,t,u^\eps,\nabla_{x,t} u^\eps) & in $M_T$, \label{e0aa}\\
    u^\eps =0 & on $\partial\mathcal O\times [-T,T]$, \label{e0ab}\\
    u^\eps\sim u^1(x,t)+\eps U_1(x,t,\phi_i/\eps) & in $ M_T \cap \{t<-T+\delta\}$ \label{e0ac}.
\end{subnumcases}
Here the functions $u^1$ and $U_1(x,t,\theta_i)$ have $(x,t)$-support strictly away from $\partial M$, the function $U_1(x,t,\theta_i)$, called the incoming profile,  is $2\pi$-periodic in $\theta_i$, and  $\phi_i(x,t)=-t+\psi_i(x)$, where $\psi_i:U\to \mathbb{R}$ is a smooth convex function that satisfies the eikonal equation
\[(\partial_{x_1}\psi_i)^2+|\nabla \psi_i|^2=1\]
on some convex open set $U\subset \mathbb{R}^n$, $U\ni x=(1,0)$.    We express condition \eqref{e0ac} by saying that the 
wave $u^\eps$  oscillates with the incoming phase $\phi_i$ in $t<-T+\delta$ (the past).  The function $f(x,t,p,q)$ is Lipschitz in the argument $(p,q)$.

The reflected phase $\phi_r$ needed to describe $u^\eps$ in the future is constructed using null bicharacteristics of $\Box$. To describe these we let $\xi=(\xi_1,\oxi)\in\mathbb{R}^n$, $\tau\in\mathbb{R}$ and write the symbol of $\Box$ as $p(x,t,\xi,\tau)=\xi_1^2+|\oxi|^2-\tau^2$.   Null bicharacteristics of $\Box$ are integral curves $\gamma:\mathbb{R}\to T^*M$ of the Hamiltonian vector field $H_p=2\xi_1\partial_{x_1}+2\oxi\partial_{\ox}-2\tau\partial_t$ such that $\gamma(s)\in p^{-1}(0)$ for all $s$.
 Thus, the components of $\gamma(s)=(x(s),t(s),\xi(s),\tau(s))$ satisfy
 \begin{align*}
\dot x(s)=2\xi,\;  \dot t(s)=-2\tau,\;  \dot \xi(s)=0,\;  \dot\tau(s)=0;\quad |\xi(s)|=|\tau(s)|.
\end{align*}
Note that null bicharacteristics are straight lines, and that the direction of $(x(s),t(s))$ is given by $(2\xi(0),-2\tau(0))$ for all $s$.
Let us set 
\begin{align*}
\xi_1^i(\ox):=\partial_{x_1}\psi_i(F(\ox),\ox), \;\oxi^i(\ox):=\partial_{\ox}\psi_i(F(\ox),\ox).
\end{align*}
For points $\ox$ such that 
\begin{align}\label{d6}
\xi^i_1(\ox)-\nabla F(\ox)\cdot \oxi^i(\ox)\leq 0
\end{align}
we define the forward incoming null bicharacteristic associated to $\phi_i$, which hits $\partial T^*M$ at $\rho:=(F(\ox),\ox,t,\xi^1_i(\ox),\oxi^i(\ox),-1)$ when $s=0$, to be\footnote{Here  ``forward" means that $t$ increases as $s$ increases. We sometimes suppress the $(\ox,t)$ dependence and write just $\gamma_i(s)$. }   
\begin{align}\label{d6z}
\gamma_i(s;\ox,t)=(F(\ox)+2s\xi_i^1(\ox), \ox+2s\oxi^i(\ox),t+2s,\xi^1_i(\ox),\oxi^i(\ox),-1),
\end{align}
and the associated forward incoming characteristic to be $\pi_{(x,t)}\gamma_i$, the projection onto the $(x,t)$ coordinates.  
In terms of $\phi_i$ the direction of  $\pi_{(x,t)}\gamma_i$ is thus $$2(\partial_{x}\phi_i(F(\ox),\ox,t),-\partial_t\phi_i(F(\ox),\ox,t)).$$
 In the case $<$ the condition \eqref{d6} means that $\pi_{(x,t)}\gamma_i$ enters $M^c:=\RR^{n+1}\setminus M$ for small $s>0$, while in the case $=$ the condition means that $\pi_{(x,t)}\gamma_i$ is tangent to $\partial M$ at $s=0$.

To identify reflected null bicharacteristics we introduce the inclusion map $i:\partial M\to M$, which induces the pullback map 
$i^*:\partial T^*M\to T^*\partial M$.  If we use the map $(\ox,t)\to (F(\ox),\ox,t)$ to parametrize $\partial M$,  then $i^*$ is given in coordinates by 
\begin{align*}
i^*(F(\ox),\ox,t,\xi_1,\oxi,\tau)=(\ox,t,\xi_1\nabla F(\ox)+\oxi,\tau).
\end{align*}
Thus, the kernel of $i^*$ is $N^*\partial M$, that is:
\begin{align}\label{d4}
\begin{split}
(\xi_1\nabla F(\ox)+\oxi,\tau)=0 \ \Leftrightarrow & \ (\xi_1,\oxi,\tau)\parallel (1,-\nabla F(\ox),0) \\ \Leftrightarrow & \ 
 (F(\ox),\ox,t,\xi_1,\oxi,\tau)\in N^*\partial M.
\end{split}
\end{align}
Now set $\phi_0(\ox,t):=\phi_i(F(\ox),\ox,t)$ and 
observe that $\gamma_i$ ``passes over" the point $$(\ox,t,\partial_{\ox,t}\phi_0(\ox,t))\in T^*\partial M$$ in the sense that 
\begin{align*}
i^*\gamma_i(0;\ox,t)=(\ox,t,\xi_1^i(\ox)\nabla F(\ox)+\oxi^i(\ox),-1)=(\ox,t,\partial_{\ox,t}\phi_0(\ox,t)):=\sigma(\ox,t).
\end{align*}
From \eqref{d4} we see  that for $\ox$ such that \eqref{d6} holds with $<$, there is another forward null bicharacteristic that passes over $\sigma(\ox,t)$ at $s=0$, namely
\begin{align*}
\gamma_r(s;\ox,t):=(F(\ox)+2s\xi^r_1(\ox), \ox+2s\oxi^r(\ox),t+2s,\xi_1^r(\ox),\oxi^r(\ox),-1),
\end{align*}
where $\xi^r(\ox)$ satisfies for some $c(\ox)\neq 0$:
\begin{align}\label{d8}
\xi^r(\ox)-\xi^i(\ox)=c(\ox)(1,-\nabla F(\ox)) \text{ and }|\xi^r(\ox)|=1.
\end{align}
The equations \eqref{d8} are readily solved to yield
\begin{equation*}\begin{split}\label{xir2}
    \xi_1^r & = \xi_1^i-\frac{2}{1+|\nabla F|^2}( \xi_1^i-\langle 
\oxi^i,\nabla F \rangle ), \\
    \oxi^r & = \oxi^i+\frac{ 2\nabla F }{ 1+|\nabla F|^2 }( \xi_1^i - \langle \oxi^i,\nabla F \rangle ).
\end{split}
\end{equation*}
Observe that
\begin{align}\label{d8z}
\xi^i_1(\ox)-\nabla F(\ox)\cdot \oxi^i(\ox)= 0
\Leftrightarrow \oxi^r(\ox)  = \oxi^i(\ox),
\end{align}
 the case where both 
$\pi_{(x,t)}\gamma_i$ and  $\pi_{(x,t)}\gamma_r$ are tangent to $\partial M$ at $(F(\ox),\ox,t)$.   
We call $\gamma_r$ the \emph{reflected} null bicharacteristic corresponding to $\gamma_i$.  
This discussion motivates the following definition.

\begin{defn}\label{gsill}
{\rm (a)} The \emph{grazing set} and the \emph{illuminable region} are respectively
\begin{align}\label{grill3}
\begin{split}
&G_{\phi_i} := \{ (F(\ox),\ox,t)| \ \langle \oxi^i,\nabla F \rangle - \xi_1^i=0 \}, \\ 
&I_- :=\{ (F(\ox),\ox,t)| \ \langle \oxi^i,\nabla F \rangle - \xi_1^i >0 \}. 
\end{split}
\end{align}
In \eqref{grill3} we require that $\ox\in \overline{B(0,R)}$ and $t\in [-T,T]$ for some positive $R,T$ that may vary from one example (that is, choice of $\mathcal{O}$ and $\phi_i$) to another.   

{\rm (b)} For $s_0>0$ the \emph{reflected flow map} is the map\footnote{Here we use the map $(\ox,t)\to (F(\ox),\ox,t)$ to identify $G_{\phi_i}\cup I_-$ with a subset of $\ox,t$ space.} 
\begin{align} \label{grill4}
\begin{split}
&Z_r:[0,s_0]\times (G_{\phi_i}\cup I_-)\to M \text{ where }\\
&Z_r(s,\ox,t):=\pi_{(x,t)}\gamma_r(s;\ox,t)=(F(\ox)+2s\xi^r_1(\ox), \ox+2s\oxi^r(\ox),t+2s).
\end{split}
\end{align}
\end{defn}

\begin{ass}\label{grill5}
 With our normalization that $F$ has a max at $\ox=0$ with $F(0)=1$, we always choose the incoming phase $\phi_i=-t+\psi_i(x)$ so that $(1,0,t)\in G_{\phi_i}$ for all $t$.
\end{ass}

\begin{rem}\label{gr5y}     
\textup{1. By \eqref{grill3} Assumption \ref{grill5} holds if and only if $\xi_1^i(0)=\partial_{x_1}\psi_i(1,0)=0$, since $\nabla F(0)=0$.   For $\gamma_i$ as in \eqref{d6z} we then have 
\begin{align}\label{gr5z}
\gamma_i(s;0,t)=(1,2s\oxi^i(0),t+2s,0,\oxi^i(0),-1).
\end{align}
The incoming characteristic $\pi_{(x,t)}\gamma_i(s;0,t)$ is then tangent to $\partial M$ at $(1,0, t)$ and lies in $T^*\mathring{M}$ for $s\neq 0$.} 

\textup{2. When $\psi_i(x)=\theta\cdot x$ for $\theta=(\theta_1,\otheta)\in S^{n-1}$ (plane waves), we have $$(1,0,t)\in G_{\phi_i}\Leftrightarrow \theta_1=0\Leftrightarrow \psi_i(x)=\ox\cdot\otheta.$$   When
$\psi_i(x)=|x-b|$ for $b=(1,\ob)\in\mathbb{R}^n$, $\ob\neq 0$ (spherical waves), we have  $$(1,0,t)\in G_{\phi_i}\Leftrightarrow b_1=1\Leftrightarrow \psi_i(x)=|(x_1-1,\ox-\ob)|.$$}
\textup{3.   For convenient reference we use \eqref{grill3} to record here that in the planar (resp. spherical) case, 
\begin{align}\label{gr5}
\ox\in\pi_{\ox}G_{\phi_i}\Leftrightarrow \nabla F(\ox)\cdot\otheta=0,\;(\text{resp., }F(\ox)-1-\nabla F(\ox)\cdot (\ox-\ob)=0).
\end{align}}
\end{rem}


\begin{ass}[Reflected flow map assumption (RFM)]\label{rfa}
There exist $s_0>0$ and positive constants $R,T$ as in Definition \ref{gsill} such that the map $Z_r$ is a $C^\infty$ diffeomorphism from $[0,s_0]\times I_-$ onto its range, which extends to a homeomorphism from $[0,s_0]\times (G_{\phi_i}\cup I_-)$ onto its range.
\end{ass}

Let us denote the ranges of $Z_r$ on $[0,s_0]\times (G_{\phi_i}\cup I_-)$ and $[0,s_0]\times I_-$ by $\mathcal{J}_r$ and $\mathcal{J}'_r$ respectively.  Write the map $Z_r^{-1}$ defined on $\mathcal{J}_r$ as 
\begin{align}\label{d9}
Z_r^{-1}(y_1,\oy,t')=(s,\ox,t)\in [0,s_0]\times (G_{\phi_i}\cup I_-).
\end{align}
The method of characteristics (see \cite{williams2022}, for example) constructs the reflected phase $\phi_r$ on $\mathcal{J}_r$  so that the following properties hold when \eqref{d9} holds:
\begin{align}\label{d10}
\begin{split}
&\phi_r(y_1,\oy,t')=\phi_i(F(\ox),\ox,t), \ d\phi_r(y_1,\oy,t')=(\xi^r(\ox),-1).
\end{split}
\end{align}
Thus $\phi_r$ is a characteristic phase (that is, $p(y,t',d\phi_r(y,t'))=0$) such that both $\phi_r$ and $d\phi_r$ are constant along
any  particular reflected characteristic $\pi_{(x,t)}\gamma_r(s;\ox,t)$.   

\Remark
When Assumption \ref{rfa} holds, the formulas \eqref{d10} show that 
\begin{align*}
\phi_r\in C^\infty(\mathcal{J}'_r) \text{ but we just have }\phi_r\in C^1(\mathcal{J}_r).
\end{align*}
In fact, $\phi_r$ generally fails to be in $C^2(\mathcal{J}_r)$; the singularity in $\phi_r$ is due to the singularity of  $Z_r^{-1}$
at $[0,s_0]\times G_{\phi_i}$. 


\subsection{Orders of grazing}
\label{decomposition}

To formulate the grazing set assumption GS, we first recall the classical decomposition
\begin{align*}
T^*\partial M\setminus 0=E\cup H\cup G
\end{align*}
into \emph{elliptic}, \emph{hyperbolic}, and \emph{glancing} sets.   
If $\sigma\in T^*\partial M\setminus 0$, we say that $\sigma$ belongs to $E$, $H$, or $G$ if the number of elements in  $(i^*)^{-1}(\sigma)\cap p^{-1}(0)$ is zero, two, or one respectively.   The sets $E$ and $H$ are conic open subsets of $T^*\partial M\setminus 0$, and $G$ is a closed conic hypersurface in  $T^*\partial M\setminus 0$.

Below we let $\beta(x,t):=x_1-F(\ox)$ denote the defining function of $\partial M$.  
\begin{defn}[\cite{melrosesjostrand1978cpam}]\label{d13}
Let $\sigma=(m,\nu)\in G$ and suppose $(i^*)^{-1}(\sigma)\cap p^{-1}(0)=\{\rho\}$, where $\rho\in T^*_mM$.   
We say $\sigma\in G^l$, the glancing set of order at least $l\geq 2$, if 
\begin{align*}
p(\rho)=0 \text{ and }H_p^j \beta(\rho)=0\text{ for }0\leq j<l.   
\end{align*}
Thus, $G=G^2\supset G^3\supset \dots\supset G^\infty$.  

We say $\sigma\in G^{l}\setminus G^{l+1}$, the set of glancing points of exact order $l$, if $\sigma\in G^l$ and $H^l_p\beta(\rho)\neq 0$.     
Consider points $\sigma\in G^{2k}\setminus G^{2k+1}$, $k\geq 1$,  such that $H^{2k}_p\beta(\rho)>0$.    When $k=1$, such a point $\sigma$ is a  classical \emph{diffractive point} as studied in \cite{melrose1975duke} or \cite{cheverry1996}.   When $k\geq 1$ we refer to $\sigma$ as a grazing (or diffractive) point of order $2k$, and we write 
\begin{align}\label{q3z}
\sigma\in G^{2k}_d\setminus G^{2k+1} \ \Leftrightarrow \ p(\rho)=0, \ H_p^j \beta(\rho)=0\text{ for }0\leq j<2k,\text{ and }H^{2k}_p\beta(\rho)>0.
\end{align}  
\end{defn}

\Remarks
1. If $\sigma\in G^{2k}_d\setminus G^{2k+1}$, let $\gamma(s)$ denote the  bicharacteristic of $p$ such that $\gamma(0)=\rho$.    Then $\gamma$ is tangent to $\partial T^*M$ at $\rho$ and lies $T^*\mathring{M}$ for small $s\neq 0$.

\noindent
2. \emph{Gliding points} of order $2k$, $\sigma\in G^{2k}_g\setminus G^{2k+1}$, are defined as in \eqref{q3z} with the single change $H^{2k}_p\beta(\rho)<0$.   If $\sigma\in G^l\setminus G^{l+1}$ for some odd $l$, we call $\sigma$ an \emph{inflection point} of order $l$.

\begin{defn}[Grazing points of order $\infty$]\label{q2a}
Let $\sigma\in G^\infty$ and suppose $(i^*)^{-1}(\sigma)\cap p^{-1}(0)=\{\rho\}$.
We say that $\sigma$ is a \emph{diffractive point of order $\infty$} and write $\sigma\in G^\infty_d$ if the 
bicharacteristic $\gamma(s)$ of $p$ such that $\gamma(0)=\rho$ lies $T^*\mathring{M}$ for small $s\neq 0$.
\end{defn}

\begin{defn}[Glancing points of diffractive type]\label{q2b}
We denote by 
\begin{align*}
\mathcal{G}_d:=\cup_{k=1}^\infty \left(G^{2k}_d\setminus G^{2k+1}\right)\cup G^\infty_d
\end{align*}
the set of \emph{glancing points of diffractive type}.  More simply, we refer to $\mathcal{G}_d$ as  the set of \emph{grazing points}.
\end{defn}

For a fixed $t\in\mathbb{R}$  let $\rho=(F(\ox),\ox,t,\xi_1^i(\ox),\oxi^i(\ox),-1)\in\partial T^*M\cap p^{-1}(0)$ and observe that\footnote{For the third equivalence we used \eqref{d8z} and the definition of $G$.} 
\begin{equation}\label{d12}
\begin{split}
(F(\ox),\ox,t)\in G_{\phi_i} 
\ \Leftrightarrow \ & \oxi_i(\ox)\cdot \nabla F(\ox)-\xi_1^i(\ox)=0 \\ 
\Leftrightarrow \
& 0=H_p\beta(\rho)=2\xi_1^i(\ox)-\nabla F(\ox)\cdot 2\oxi^i(\ox) \\ 
\Leftrightarrow \ & i^*\rho=(\ox,t,\xi^i_1(\ox)\nabla F(\ox)+\oxi^i(\ox),-1)\\
 & =(\ox,t,d\phi_0(\ox,t)):=\sigma\in G.
\end{split}
\end{equation}




The next proposition is a simple consequence of the strict convexity of $\mathcal{O}$.  

\begin{prop}\label{d12u}
Let  $\beta(x,t)=x_1-F(\ox)$ and let 
$$\urho=(1,0,t,\xi_1^i(0),\oxi^i(0),-1)\text{ and }\usigma=i^*\urho=(0,t,\xi^i_1(0)\nabla F(0)+\oxi^i(0),-1).$$
Under  Assumptions \ref{n1z} (strict concavity of $F$) and \ref{grill5} ($(1,0,t)\in G_{\phi_i}$), we have
$\usigma\in\mathcal{G}_d$.  

\end{prop}

\begin{proof}
By \eqref{d12} with $\ox=0$ we have $H_p\beta(\urho)=0$.  For $\gamma_i(s):=\gamma_i(s;0,t)$ as in \eqref{d6z} we have 
\begin{align}\label{d12z}
H_p^\ell\beta(\gamma_i(s))=\left(\frac{d}{ds}\right)^\ell \beta(\gamma_i(s))\text{ for }\ell\in\mathbb{N}.
\end{align}
If $H_p^\ell(\beta)(\urho)=0$ for all $\ell\geq 2$, then $\usigma\in G^\infty_d$, since strict convexity of $\mathcal{O}$ implies $\gamma_i$ lies in  $T^*\mathring{M}$ for $s\neq 0$.\footnote{Recall Remark \ref{gr5y}(1).}

Otherwise, let $m$ be the smallest exponent for which $H_p^m\beta(\urho)\neq 0$.  With \eqref{d12z} we see that $m=2k$ must be even and $H_p^{2k}\beta(\urho)>0$, since otherwise $\gamma_i$ enters $\mathcal{O}$ for some $s\neq 0$ small.  Thus, by \eqref{q3z} $\usigma\in G_d^{2k}\setminus G^{2k+1}.$
\end{proof}

\begin{rem}\label{d12x}
\textup{The proof of Proposition \ref{d12u} shows more generally that if $$\sigma=(\ox,t,\xi^i_1(\ox)\nabla F(\ox)+\oxi^i(\ox),-1)\in G,$$ then   
$\sigma$ must lie in $\mathcal{G}_d$ since 
$\gamma_i(s,\ox,t)\text{ lies in }T^*\mathring{M} \text{ for }s\neq 0.$}


\end{rem}


\begin{ass}[Grazing set assumption (GS)]\label{gs}
Let  $\beta$, $\usigma\in\mathcal{G}_d$, and $\urho$ be as in Proposition \ref{d12u}.  
 There exists an open set $U\ni (1,0,t)$ in $\mathbb{R}^n\times \mathbb{R}_t$ and a $C^1$ function $\zeta:U\to \mathbb{R}$ with $d\zeta\neq 0$ on $U$ such that:

{\rm (a)} $G_{\phi_i}\cap U=\{(x,t)\in U: \zeta(x,t)=0 \text{ and }\beta(x,t)=0\}$;

{\rm (b)} $H_p\zeta(\urho)\neq 0$;

{\rm (c)} There exists an open set $V\ni\usigma$ in $T^*\partial M$ such that $G\cap \textrm{Graph }d\phi_0\cap V\subset \mathcal{G}_d$.\footnote{Recall $\phi_0=\phi_i\circ i:\partial M\to \mathbb{R}$.}
\end{ass}

Thus, the assumption requires that $G_{\phi_i}$ be a $C^1$ codimension two submanifold of $\mathbb{R}^n\times \mathbb{R}_t$ near $(1,0,t)$.\footnote{Part (b) implies $d\zeta\wedge d\beta\neq 0$ at $(1,0,t)$.}  In the  results of \S \ref{gsa} where the GS assumption holds, conditions (b) and  (c) are seen to follow from the strict convexity of $\mathcal{O}$.

For $\usigma\in G$ as in \eqref{d12} it will be  convenient to have a simple reformulation of the condition  that $\usigma\in G^{2k}_d\setminus G^{2k+1}$  in terms of $F$.   Let 
$$\gamma_i(s;0,t)=(1+2s\xi^i_1(0),2s\oxi^i(0),t+2s,\xi_1^i(0),\oxi^i(0),-1)$$
be the incoming null bicharacteristic such that $\gamma_i(0;0,t)=\urho$.   For $F$ as in Assumption \ref{n1z} write the Taylor expansion of $F$ at $0$ as\footnote{Strict concavity of $F$ implies $p_1(\ox)=0$.} 
\begin{align}\label{TF}
F(\ox)=1+p_2(\ox)+p_3(\ox)+\cdots+p_{2k}(\ox)+O(|\ox|^{2k+1}),
\end{align}
where $p_j$ is a homogeneous polynomial of degree $j$.  With $\beta(x,t)=x_1-F(\ox)$ we have
\begin{align}
\beta(\gamma_i(s))=2s\xi^i_1(0)-\left[p_2(2s\oxi^i(0))+\cdots+p_{2k}(2s\oxi^i(0))+O(s^{2k+1})\right].
\end{align}
A direct computation using $H_p^j\beta(\gamma_i(s))=\left(\frac{d}{ds}\right)^j\beta(\gamma_i(s))$   shows that
\begin{align}\label{d14}
\begin{split}
&\mathrm{(a)} \; H_p^j\beta(\urho)=0 \text{ for }1\leq j<2k\text{ and }H_p^{2k}\beta(\urho)>0 \text{ if and only if }\\
&\mathrm{(b)} \; \xi^i_1(0)-\nabla F(0)\cdot\oxi^i(0)=\xi^i_1(0)=0, \;p_j(\oxi^i(0))=0 \text{ for }1<j<2k,\\
& \qquad \text{and } p_{2k}(\oxi^i(0))<0.
\end{split}
\end{align}

\begin{rem}\label{d14z}
\textup{When $\phi_i(x,t)=-t+|x-b|$ with $b=(1,\ob)$, $\ob\neq 0$,  the set $\pi_x G_{\phi_i}$ is the intersection of the smooth hypersurfaces 
$S_1=\{x:\mathcal{H}_1(x):=x_1-F(\ox)=0\}$ and $S_2=\{x:\mathcal{H}_2(x):=x_1-1-\nabla F(\ox)\cdot (\ox-\ob)=0\}$.
A normal vector to $S_1$ (resp. $S_2$) at $x=(1,0)$ is $n_1=(1,0)$ (resp.  $n_2=(1,\nabla^2 F(0)\ob)$).   Writing $F$ as in \eqref{TF}, we see that $\nabla^2F(0)\ob=\nabla^2p_2(0)\ob$.   We claim\footnote{The first equivalence follows from the Taylor expansion of $p_2$.   To see $\Leftarrow$ in the second equivalence observe that if $p_2(\ob)=0$, then $p_2(\ox)$ has the form $c_1(\ob^\perp\cdot\ox)^2$ or $c_2(\ob^\perp\cdot\ox)(\ell\cdot\ox)$ for some $\ell\in \mathbb{R}^{n-1}$ such that $\ell\cdot \ob\neq 0$ and real $c_i$.   If the second form holds with $c_2\neq 0$, then $F$ is not strictly concave, so $p_2=c_1(\ob^\perp\cdot\ox)^2$ for some $c_1\leq 0$ and hence $\nabla^2p_2(0)\ob=0$.}
\begin{align}\label{d14y}
p_2(\ob)\neq 0\Leftrightarrow \langle\nabla^2p_2(0)\ob,\ob\rangle\neq 0\Leftrightarrow \nabla^2p_2(0)\ob\neq 0.
\end{align}
With \eqref{d14} this implies that $n_1$ and $n_2$ are linearly independent if and only if $\usigma=(0,t,-\frac{\ob}{|\ob|},-1)\in G^2_d\setminus G^3.$  This is the \emph{only} case where  the implicit function theorem can be applied directly to $\mathcal{H}_1$ and $\mathcal{H}_2$ to show that the GS assumption holds for incoming spherical waves.}   

\textup{When $\phi_i(x,t)=-t+\otheta\cdot \ox$, the set $\pi_{\ox}G_{\phi_i}$ is defined by $\nabla F(\ox)\cdot\otheta=0$.  The argument for \eqref{d14y} shows that $\nabla^2F(0)\otheta\neq 0\Leftrightarrow \usigma=(0,t,\otheta,-1)\in G^2_d\setminus G^3$, so  this is the only case where the implicit function theorem applies directly to verify the GS assumption for incoming plane waves.}
\end{rem}

\subsection{Main results}

\S \ref{siw} proves a strengthened, global version of the reflected flow map assumption (RFM) for spherical incoming waves.

\begin{theo}
Let the convex obstacle $\mathcal{O}\subset \mathbb{R}^n$ be as Assumption \ref{n1z}, and consider incoming waves associated to the spherical phase function 
$\phi_i(x,t)=-t+|x-b|$, where $b\in \mathbb{R}^n\setminus \overline{\mathcal{O}}$.  In the definition of the grazing set $G_{\phi_i}$ and the illuminable region $I_-$, Definition \ref{grill}, allow $R$ and $T$ to be any positive numbers.   Then the RFM assumption, Assumption \ref{rfa}, holds with $s_0$ any positive number.  
\end{theo}

The main result of \S \ref{ciw} is a verification of the RFM assumption for general convex incoming waves.

\begin{theo}
Let the convex obstacle  $\mathcal{O}\subset \mathbb{R}^n$ be as \eqref{defo}, and consider incoming waves associated to any incoming phase function of the form $\phi_i(x,t)=-t+\psi_i$, where $\psi_i:\Omega\to \mathbb{R}$ is $C^\infty$, convex, and satisfies 
$$(\partial_{x_1}\psi_i)^2 + |\nabla\psi_i|^2=1$$ on some $\mathbb{R}^n$-open set $\Omega\ni (1,0)$.   Then the RFM assumption, Assumption \ref{rfa}, holds with $s_0$ and $T$ any positive numbers provided $R$ is taken sufficiently small.\footnote{Here again $R$ and $T$ are the parameters appearing in Definition \ref{grill}.}

\end{theo}

Unlike the RFM assumption the GS assumption becomes more difficult to study and is more likely to fail as the dimension $n$ of the obstacle increases.  
We know from \cite{ww2023} that for 2D strictly convex obstacles the GS assumption \emph{always} holds for incoming plane waves, and we show in \S \ref{natureg} that is also true for incoming spherical waves.   

One might hope that the GS assumption is always satisfied for strictly convex obstacles of dimension $n\geq 3$; after all, the obstacles have $C^\infty$ boundaries and the incoming waves are planar or spherical.   But this is not the case.    For both plane and spherical waves we provide  examples in section \S\ref{bad}  where the GS assumption fails for 3D convex obstacles.     There is reason to expect the situation to be worse for $n$D obstacles when $n\geq 4$.   In Proposition \ref{rnd} of section \S\ref{nd} we identify a class of $n-$dimensional obstacles with special symmetry for which the GS assumption holds.  Along with the 2D case, these $nD$  examples include cases of infinite order tangency.

We are mainly concerned in \S \ref{3do} to provide large classes of examples where the GS assumption holds in dimension $n=3$.  
Our main positive result on the GS assumption for 3D obstacles is the following theorem, which is stated more carefully as Theorem \ref{u1} of \S\ref{3do}.

\begin{theo}\label{u1xx}
 Let the obstacle be defined as in Assumption \ref{n1z} by $x_1=F(\ox)$.   Suppose  for some $k\in\mathbb{N}$ that 
 \begin{align}\label{u1y}
F(\ox)=1-G_{2k}(\ox)+r(\ox),
\end{align}
 where
$G_{2k}$ is a homogeneous real polynomial of degree $2k$ in $\ox=(x_2,x_3)$ such that 
\begin{align}\label{u1yy}
\nabla^2 G_{2k}(\ox)>0 \text{ for }\ox\neq 0,
\end{align}
 and $r(x)=O(|\ox|^{2k+1})$ is a $C^\infty$ remainder term.
Take the incoming phase 
to be $\phi_i=-t+|x-b|$, $b=(1,\ob)$ with $\ob\in\mathbb{R}^2\setminus \{0\}$.  Then the point $\usigma=i^*(1,0,t,0,-\frac{\ob}{|\ob|},-1)\in G_d^{2k}\setminus G^{2k+1}$ and the  GS assumption, Assumption \ref{gs}, is satisfied.


\end{theo}

The next theorem, stated more carefully as Theorem \ref{A} of \S\ref{branch},  rules out \emph{branching} for a class of convex obstacles and spherical incoming waves for which grazing can occur to any finite or infinite order at $x=(1,0)$;  condition \ref{u1yy} need not hold.  Branching occurs when $\pi_{x}G_{\phi_i}$ consists of more than two smooth curves which meet and terminate at $x=(1,0)$.  In any such case the GS assumption fails of course.  

\begin{theo}
Let the obstacle $\mathcal{O}$ be defined as in Assumption \ref{n1z} by $x_1=F(\ox)$, and assume that\footnote{These conditions mean that $\partial{\mathcal{O}}$ is strictly convex with $x=(1,0)$ being an isolated zero of the Gauss curvature.} 
\begin{align}\label{u1yz}
\nabla^2F(\ox)<0\text{ for }\ox\neq 0 \text{ when }\ox\text{ is near }0.
\end{align}
Then branching does not occur for incoming spherical waves.  The set  $\pi_{x}G_{\phi_i}$ is a single continuous curve through $x=(1,0)$.

\end{theo}

Even when branching does not occur, the GS assumption may fail when the set $\pi_{\ox}G_{\phi_i}$ has a \emph{cusp}; such examples are given in \S \ref{bad} for both spherical and planar incoming waves.  
These examples illustrate  different ways  the conclusion of Theorem \ref{u1xx} can fail if \eqref{u1yy} does not hold, that is, if the lowest order nonconstant homogeneous polynomial in the Taylor expansion of $F(\ox)$ at $\ox=0$ does not satisfy \eqref{u1yy}.

\subsection{Applications}
 The results of this paper extend the range of application of the main result of \cite{ww2023}.  That result, Theorem 2 of \cite{ww2023}, gives an explicit geometric optics description of solutions to the initial boundary value problem \eqref{e0aa}--\eqref{e0ac} whenever the RFM and GS assumptions are satisfied by the obstacle $\mathcal{O}$ and incoming phase $\phi_i$.     In \cite{ww2023} we verified these assumptions 
 for some families of convex obstacles $\mathcal{O}$, but only for \emph{planar} incoming waves associated to linear phases 
 $\phi_i(x,t)=-t+\theta\cdot x$, $|\theta|=1$.    The results of \S\S \ref{siw}, \ref{natureg}--\ref{3do} of this paper show that both assumptions hold for certain families of convex obstacles with grazing points of arbitrarily high order,
when the incoming phases are the nonlinear phases that describe spherical waves, $\phi_i(x,t)=-t+|x-b|$.\footnote{When $\usigma\in G^2_d\setminus G^3$, the results of \cite{cheverry1996} and \cite{ww2023} show that the RFM and GS assumptions are satisfied for any incoming characteristic phase; moreover, the grazing set is $C^\infty$, not just $C^1$.}

We now restate, somewhat informally, the main result of \cite{ww2023} for the problems considered here, referring the reader to that paper for a  more precise and general formulation.

\begin{theo}[{\cite[Theorem 2]{ww2023}}]
Consider the initial boundary value problem \eqref{e0aa}--\eqref{e0ac} under the RFM assumption, Assumptions \ref{rfa}.  Suppose the initial oscillatory profile $U_1(x,t,\theta_i)$ is supported near a characteristic of $\phi_i$ that grazes $\partial M$  (possibly to high order) at the point $(x_1,\ox,t)=(1,0,0)\in G_{\phi_i}$.  Assume that the GS assumption, Assumption \ref{gs}, is satisfied near $\usigma=(0,0,d\phi_i(0,0))\in \mathcal{G}_d$.  Then if $T>0$ is small enough, the solution $u^\eps\in H^1(M_T)$ satisfies 
 \begin{align}\label{d15}
u^\eps(x,t)|_{M_T}\sim_{H^1} u(x,t)+\eps U_r(x,t,\phi_r/\eps)+\eps U_i(x,t,\phi_i/\eps).
\end{align}
Here $U_k(x,t,\theta_k)$ for $k=r,i$  is periodic in $\theta_k$ with mean zero,
and the functions 
\begin{align}\label{d16}
u\in H^1(M_T),\; \partial_{\theta_r}U_r\in L^2(M_T\times \mathbb{T}),\;  \partial_{\theta_i}U_i\in  L^2(M_T\times \mathbb{T})
\end{align}
  are constructed to satisfy the profile equations \cite[(4.4)--(4.6)]{ww2023}.  

\end{theo}

\Remarks
1. The meaning of ``$\sim_{H^1}$" in \eqref{d15} is explained in \cite[Definition 1.3]{ww2023}. Roughly, it means that the leading part of $u^\eps$ has the given form up to an error that is small in $H^1$ when  the wavelength $\eps$ is small.

\noindent
2. The profile equations referred to in the theorem are a coupled system for the functions $u$, $U_r$, $U_i$ in \eqref{d16} in which the equations for $U_r$, $U_i$ describe transport along the characteristic vector fields associated to $\phi_r$ and $\phi_i$:
$$T_{\phi_k}=2\nabla_x\phi_k\partial_x-2\partial_t\phi_k\partial_t,\;\;k=r,i.$$

\noindent
3. We have seen that the RFM assumption allows us to construct the reflected phase $\phi_r$ as a $C^1$ function on $\mathcal{J}_r$, the range of $Z_r$ as in \eqref{grill4}.  The GS assumption implies that the union of the forward and backward flowouts of the grazing set $G_{\phi_i}$ under $T_{\phi_i}$ is a $C^1$ hypersurface in $\mathbb{R}^{n+1}_{x,t}$.   We write this union as $SB=SB_+\cup SB_-$ and refer to $SB_+$, the forward flowout, as the ``shadow boundary".   From the profile equations \cite[(4.4)--(4.6)]{ww2023} we can read off the property that $U_r(x,t,\theta_r)$ and $U_i(x,t,\theta_i)$ have their supports inside the set of points of $M_T$ that can be reached by starting from points $(x,t)$ in the support of the initial profile $U_1(x,t,\theta_i)$ as in  \eqref{e0ac}, then flowing forward along $T_{\phi_i}$, and continuing along  $T_{\phi_r}$ in any case where a characteristic of $T_{\phi_i}$ hits the boundary. 
With \eqref{d15} this shows that the exact solution $u^\eps$ has no high frequency oscillations in the shadow region adjacent to $SB_+$ that are detectable in the $H^1$ norm.


\section{The reflected flow map for spherical incoming waves}\label{siw}

As in \S \ref{dar}, we assume $\mathcal O$ is a convex obstacle given by 
\[ \mathcal O:=\{ (x_1, \ox) \mid x_1< F(\ox),\; \ox\in\RR^{n-1} \} \]
where $F: \RR^{n-1}\to \RR$ is a $C^\infty$ strictly concave function.
 \footnote{The normalization that $F$ has a max at $\ox=0$ with $F(0)=1$ is actually not used in \S \ref{siw}, but the reader is free to assume that.}  
We consider the following phase function for spherical incoming waves
\[ \phi_i(x,t) = -t+|x-b|, \ b\in \RR^n\setminus \overline{\mathcal{O}}. \]
We define 
\[ \rho(\ox):=|(F(\ox),\ox)-b| \]
and 
\begin{equation}\label{def:alpha}
\alpha(\ox):=(\alpha_1(\ox), \overline{\alpha}(\ox))\in \mathbb S^{n-1}, \ \alpha_1(\ox):=\frac{F(\ox)-b_1}{\rho(\ox)}, \ \overline{\alpha}(\ox):=\frac{\ox-\overline{b}}{\rho(\ox)}. 
\end{equation}
According to Definition \ref{gsill}, the grazing set and the illuminable regions are given by 
\begin{align}\label{grill}
    \begin{split}
    G_{\phi_i}= & \{ (F(\ox),\ox,t) \mid \alpha_1(\ox)=\langle \nabla F(\ox), \oalpha(\ox) \rangle \},\\
    I_-= & \{ (F(\ox),\ox,t) \mid \alpha_1(\ox)<\langle \nabla F(\ox), \oalpha(\ox) \rangle \}.
\end{split}
\end{align}
Fix any $R>0$ and any $T>0$.  In \eqref{grill} and throughout \S \ref{siw} we allow $\ox\in \overline{B(0,R)}\subset \RR^{n-1}$ and $t\in [-T,T]$.

For $t^0\in \RR$, $x^0\in \RR^n\setminus (\mathcal O\cup\{b\})$,
we have 
\[ \nabla_{x,t}\phi_i(x^0,t^0) = \left( \frac{x^0-b}{|x^0-b|},-1 \right). \]
The incoming bicharacteristic passing $(x^0,t^0,\frac{x^0-b}{|x^0-b|}, -1)$ is
\[ \gamma_i(s)=\left(x_1^0+2s\frac{x_1^0-b_1}{|x^0-b|}, \ox^0+2s\frac{\ox^0-\ob}{|x^0-b|}, t^0-2s\tau, \frac{x_1^0-b_1}{|x^0-b|}, \frac{\ox^0-\ob}{|x^0-b|},-1  \right). \]
Suppose $\gamma_i(s)$ intersects $\partial T^*((\RR^n\setminus \mathcal{O}) \times \RR)$ at $(F(\ox),\ox, t', \xi^i_1, \oxi^i, -1)$, then $\xi^i(\ox) = \frac{x^0-b}{|x^0-b|}=\alpha(\ox)$. Let $\xi^r(\ox)$ be the reflected direction, then 
\[ (\xi^i, -1)-(\xi^r,\tau^r) = c(1,-\nabla F,0), \ |\xi^r|=|\tau^r|. \]
From here we solve 
\begin{equation}\begin{split} \label{b2}
\xi_1^r(\ox) = & \alpha_1(\ox)+\frac{2}{1+|\nabla F(\ox)|^2}\left( 
\langle \nabla F(\ox),\oalpha(\ox) \rangle-\alpha_1(\ox) \right), \\
\oxi^r(\ox) = & \oalpha(\ox)-\frac{2\nabla F(\ox)}{1+|\nabla F(\ox)|^2}\left( \langle \nabla F(\ox),\oalpha(\ox) \rangle-\alpha_1(\ox) \right).
\end{split}\end{equation}
The reflected flow map is then given by 
\[ Z_r(s,\ox,t)=(F(\ox)+2s\xi_1^r(\ox), \ox+2s\oxi^r(\ox), t+2s ). \]

\begin{prop}\label{b1}
    Fix any $s_0>0$ and take $\ox\in \overline{B(0,R)}$, $t\in [-T,T]$ for any positive $R$ and $T$.  The reflected flow map $Z_r$ is a diffeomorphism from $[0,s_0]\times I_-$ onto its range and extends to a homeomorphism from $[0,s_0]\times (G_{\phi_i}\cup I_-)$ onto its range.    
\end{prop}

\begin{proof}
It suffices to show the following two claims
\begin{enumerate}
    \item[1.] $Z_r$ is injective on $[0,s_0]\times (G_{\phi_i}\cup I_-)$;
    \item[2.] The Jacobian $j$ of $Z_r$ is non-zero on $[0,s_0]\times I_-$.
\end{enumerate}

\noindent
{\bf 1. Injectivity.}
We claim that the following holds when $\ox^*\neq \ox$
\[ \langle \xi^r(\ox^*)-\xi^r(\ox), (F(\ox^*)-F(\ox), \ox^*-\ox) \rangle > 0. \]
Indeed, using \eqref{b1} we have 
\begin{equation}\label{x1}
\oxi^r(\ox) = \oalpha(\ox) - \nabla F(\ox)( \xi_1^r(\ox) - \alpha_1(\ox) ) 
\end{equation}
and the same equality holds when $\ox$ is replaced by $\ox^*$.
Equation \eqref{x1} implies that
\[\begin{split}
    &\langle \oxi^r(\ox^*)-\oxi^r(\ox), \ox^*-\ox \rangle\\
    = & \langle \oalpha(\ox^*)-\oalpha(\ox), \ox^*-\ox \rangle\\
    & + \langle \nabla F(\ox^*), \ox-\ox^* \rangle (\xi_1^r(\ox^*)-\alpha_1(\ox^*)) 
     + \langle \nabla F(\ox), \ox^*-\ox \rangle(\xi_1^r(\ox)-\alpha_1(\ox))\\
    \geq & \langle \oalpha(\ox^*)-\oalpha(\ox), \ox^*-\ox \rangle \\
    & + (F(\ox)-F(\ox^*))(\xi_1^r(\ox^*) - \alpha_1(\ox^*)) + (F(\ox^*)-F(\ox))(\xi_1^r(\ox)-\alpha_1(\ox))\\
    = & -(F(\ox^*)-F(\ox))(\xi_1^r(\ox^*)-\xi_1^r(\ox))\\
    & + (F(\ox^*)-F(\ox))(\alpha_1(\ox^*)-\alpha_1(\ox))+\langle \oalpha(\ox^*)-\oalpha(\ox), \ox^*-\ox \rangle.
\end{split}\]
Since $F$ is strictly concave, we know the equality holds if and only if 
\begin{equation}\label{eq1}
\xi_1^r(\ox^*)=\alpha_1(\ox^*) \text{ and } \xi_1^r(\ox)=\alpha_1(\ox).
\end{equation}
Next we show that 
\[ (F(\ox^*)-F(\ox))(\alpha_1(\ox^*)-\alpha_1(\ox))+\langle \oalpha(\ox^*)-\oalpha(\ox), \ox^*-\ox \rangle \geq 0, \]
that is,
\[ \langle \alpha(\ox^*)-\alpha(\ox), (F(\ox^*),\ox^*)-(F(\ox),\ox) \rangle \geq 0. \]
Using definitions of $\alpha$ and $\rho$, we need to show 
\[ \langle \alpha(\ox^*)-\alpha(\ox), \rho(\ox^*)\alpha(\ox^*)-\rho(\ox)\alpha(\ox)  \rangle \geq 0. \]
A direct computation shows that the left-hand side is
\[ (\rho(\ox^*)+\rho(\ox))(1-\langle \alpha(\ox^*), \alpha(\ox) \rangle) \]
which is non-negative and equals zero if and only if \begin{equation}\label{eq2}
\alpha(\ox^*)=\alpha(\ox).
\end{equation}
Thus we have shown that 
\[ \langle \xi^r(\ox^*)-\xi^r(\ox), ( F(\ox^*)-F(\ox), \ox^*-\ox ) \rangle\geq 0 \]
and equality holds if and only if both \eqref{eq1} and \eqref{eq2} hold. It remains to show that the equality can not hold. Otherwise, combining \eqref{eq1} and \eqref{eq2}, we know
\[ \alpha(\ox^*)=\alpha(\ox) \text{ and } \xi^r(\ox^*)=\xi^r(\ox). \]
But this implies that 
\[ \nabla F(\ox^*) = \nabla F(\ox) \]
which is impossible when $\ox^*\neq \ox$.


The map $Z_r$ is continuous and injective on the compact domain $[0,s_0]\times(G_{\phi_i}\cup I_-)$, and hence is a homeomorphism on this domain.

\noindent
{\bf 2. Diffeomorphism.}
As in the plane wave case \cite[Proposition 8.10]{ww2023}, we have 
\[ j(s,\ox,t) = 2\xi_1^r\det(A), \ A:= I-\frac{\oxi^r \otimes \nabla F}{\xi_1^r} +2s \left( \frac{\partial\oxi^r}{\partial \ox} -\frac{\oxi^r \otimes \nabla \xi_1^r}{\xi_1^r}  \right). \]

We first differentiate the identity $(\xi_1^r)^2+|\oxi^r|^2=1$ with respect to $\ox$
to obtain 
\[ \xi_1^r\nabla \xi_1^r +\oxi^r\frac{\partial \oxi^r}{\partial\ox}=0 \ \Rightarrow\ \nabla \xi_1^r = -\frac{\oxi^r}{\xi_1^r}\frac{\partial\oxi^r}{\partial\ox}. \]
Thus 
\[ \frac{\partial\oxi^r}{\partial\ox} -\frac{\oxi^r\otimes \nabla \xi^r_1}{\xi_1^r} = \left(I+\frac{\oxi^r\otimes \oxi^r}{(\xi_1^r)^2}\right)\frac{\partial \oxi^r}{\partial \ox}. \]
In the following we denote 
\[ B:=I-\frac{ \oxi^r\otimes \nabla F }{\xi_1^r}, \ C:=I+\frac{\oxi^r\otimes \oxi^r}{(\xi_1^r)^2}. \]
Then 
\[ A=B+2sC\frac{\partial \oxi^r}{\partial \ox}. \]

\noindent
{\em 2.1 Computation of $\frac{\partial \oxi^r}{\partial\ox}$.}
Differentiate \eqref{x1} with respect to $\ox$ and we find 
\begin{equation}\label{dxr1}
\frac{\partial \oxi^r}{\partial\ox} = \frac{\partial \oalpha}{\partial \ox}+(\alpha_1-\xi_1^r)\nabla^2 F +\nabla F\otimes \nabla (\alpha_1-\xi_1^r). 
\end{equation}
Differentiating the first equation in \eqref{b2} with respect to $\ox$ gives 
\begin{equation}\label{dxr2}
\nabla (\alpha_1-\xi_1^r) = -\frac{2\oxi^r}{1+|\nabla F|^2}\nabla^2 F - \frac{2}{1+|\nabla F|^2}\left( \nabla F+\frac{\oalpha}{\alpha_1} \right)\frac{\partial\oalpha}{\partial\ox}. 
\end{equation}
In the computation we used $|\alpha|=1$, which implies that
\[ \nabla \alpha_1 = -\frac{\oalpha}{\alpha_1}\frac{\partial \oalpha}{\partial \ox}. \]
Putting \eqref{dxr1} and \eqref{dxr2} together gives 
\[ \frac{\partial\oxi^r}{\partial\ox}=\left( 
I-\frac{ 2\nabla F\otimes (\alpha_1\nabla F+\oalpha) }{\alpha_1(1+|\nabla F|^2)} \right)\frac{\partial\oalpha}{\partial \ox} + \left( (\alpha_1-\xi_1^r)I-\frac{ 2\nabla F\otimes \oxi^r }{ 1+|\nabla F|^2 } \right) \nabla^2 F. \]
In the following we denote
\[ K:=\left( 
I-\frac{ 2\nabla F\otimes (\alpha_1\nabla F+\oalpha) }{\alpha_1(1+|\nabla F|^2)} \right)\frac{\partial\oalpha}{\partial \ox}, \ L:= \left( (\alpha_1-\xi_1^r)I-\frac{ 2\nabla F\otimes \oxi^r }{ 1+|\nabla F|^2 } \right) \nabla^2 F.\]
A direct computation using the definition \eqref{def:alpha} of $\oalpha$ gives
\[ \frac{\partial\oalpha}{\partial \ox} = \frac{I-\oalpha\otimes (\alpha_1\nabla F+\oalpha)}{\rho}. \]
Inserting this in the expression of $K$ gives
\[\begin{split}
K = &  \tfrac{1}{\rho}\left( 
I-\tfrac{ 2\nabla F\otimes (\alpha_1\nabla F+\oalpha) }{\alpha_1(1+|\nabla F|^2)} \right)\left( I-\oalpha\otimes (\alpha_1\nabla F+\oalpha) \right)\\
= & \tfrac{1}{\rho}\left( I - \tfrac{ 2\nabla F\otimes (\alpha_1\nabla F+\oalpha) }{\alpha_1(1+|\nabla F|^2)} - \oalpha\otimes (\alpha_1\nabla F+\oalpha) + \tfrac{ 2\langle 
\alpha_1\nabla F+\oalpha, \oalpha \rangle \nabla F\otimes (\alpha_1\nabla F+\oalpha) }{\alpha_1(1+|\nabla F|^2)} \right)\\
= & \tfrac{1}{\rho}\left( I - \tfrac{ 2\nabla F\otimes (\alpha_1\nabla F+\oalpha) }{\alpha_1(1+|\nabla F|^2)} - \oalpha\otimes (\alpha_1\nabla F+\oalpha) + \tfrac{ 2(1+\alpha_1(\langle \nabla F,\oalpha \rangle-\alpha_1)) \nabla F\otimes (\alpha_1\nabla F+\oalpha) }{\alpha_1(1+|\nabla F|^2)} \right)\\
= & \tfrac{1}{\rho}\left( I  - \oalpha\otimes (\alpha_1\nabla F+\oalpha) + \tfrac{ 2(\langle \nabla F,\oalpha \rangle-\alpha_1) \nabla F\otimes (\alpha_1\nabla F+\oalpha) }{(1+|\nabla F|^2)} \right)\\
= & \tfrac{ I-\oxi^r\otimes (\alpha_1\nabla F+\oalpha) }{\rho}.
\end{split}\]

Now for $L$, we write 
\[ L=(\alpha_1-\xi_1^r)\left( I - \frac{2\nabla F\otimes \oxi^r}{(\alpha_1-\xi_1^r)(1+|\nabla F|^2)} \right)\nabla^2 F. \]
Use the first equation in \eqref{b2} and we have
\[ L=\left( I+\frac{ \nabla F\otimes \oxi^r }{ \langle \nabla F,\oalpha \rangle - \alpha_1 } \right) \frac{2(\langle \nabla F,\oalpha \rangle - \alpha_1)}{1+|\nabla F|^2}\times (-\nabla^2F). \]

We need the following lemma to proceed.
\begin{lemm}\label{lem:det}
Let $a,b\in \RR^{n-1}$ be row vectors, then
\[ \det(I+a\otimes b) = 1+\langle a,b \rangle. \]
Moreover,
    \[ (I+a\otimes b)^{-1} = I-\frac{a\otimes b}{1+\langle 
a,b \rangle} \]
as long as $\langle a,b \rangle\neq -1$.
\end{lemm}
\begin{proof}
For the first part, see \cite[(8.50)]{ww2023}. For the second identity, we compute directly
\[\begin{split} 
(I+a\otimes b)\left( I-\frac{a\otimes b}{1+\langle a,b \rangle} \right) 
= I+\left( 1-\frac{1}{1+\langle a,b \rangle} \right)a\otimes b -\frac{\langle a,b \rangle a\otimes b}{1+\langle a,b \rangle} =I.
\end{split}\]
This completes the proof.
\end{proof}

Apply Lemma \ref{lem:det} to $B$ and we find that
\[ \det(B) = \frac{\xi_1^r-\langle \oxi^r,\nabla F \rangle}{\xi_1^r} = \frac{ \langle \nabla F, \oalpha \rangle - \alpha_1 }{ \xi_1^r }.\]
Moreover,
\[ B^T = I-\frac{\nabla F\otimes \oxi^r}{\xi_1^r}, \ (B^T)^{-1} = I + \frac{ \nabla F\otimes \oxi^r }{ \xi_1-\langle \nabla F, \oxi^r \rangle } = I + \frac{ \nabla F\otimes \oxi^r }{ \langle 
\nabla F, \oalpha \rangle - \alpha_1 }. \]

As a result, we have\footnote{For the last equality in \eqref{qz1} we used $\xi_1\nabla F+\oxi^r=\alpha_1\nabla F+\oalpha^r$. }
\begin{align}\label{qz1}
\begin{split} 
B^T K = & \tfrac{1}{\rho}\left( I-\tfrac{\nabla F\otimes \oxi^r}{\xi_1^r} \right)\left( I-\oxi^r\otimes (\alpha_1\nabla F+\oalpha) \right) \\
= & \tfrac{1}{\rho}\left( I -\tfrac{\nabla F\otimes \oxi^r}{\xi_1^r}- \oxi^r\otimes (\alpha_1\nabla F+\oalpha) +\tfrac{|\oxi^r|^2 \nabla F\otimes (\alpha_1\nabla F+\oalpha) }{\xi_1^r} \right)\\
= & \tfrac{1}{\rho}\left( I -\tfrac{\nabla F\otimes \oxi^r}{\xi_1^r}- \oxi^r\otimes (\alpha_1\nabla F+\oalpha) +\tfrac{(1-(\xi_1^r)^2) \nabla F\otimes (\alpha_1\nabla F+\oalpha) }{\xi_1^r} \right)\\
= & \tfrac{1}{\rho}\left( I+\tfrac{ \nabla F\otimes (\alpha_1\nabla F+\oalpha-\oxi^r) }{\xi_1^r} - (\xi_1\nabla F+\oxi^r)\otimes (\alpha_1\nabla F+\oalpha) \right)\\
= & \tfrac{ I+\nabla F \otimes \nabla F - (\alpha_1\nabla F+\oalpha) \otimes (\alpha_1\nabla F+\oalpha) }{\rho}.
\end{split}
\end{align}
From here we can see that $B^TK$ is symmetric. We claim that $B^TK$ is also positive definite. Indeed, for any $v\in \RR^{n-1}$, there holds
\[\begin{split} 
\rho\langle B^TKv,v \rangle = & |v|^2+|\langle \nabla F, v \rangle|^2 -|\langle \alpha_1\nabla F+\oalpha, v \rangle|^2 \\
= & |v|^2+|\langle \nabla F, v \rangle|^2 -  | \langle (\alpha_1,\oalpha), ( \langle \nabla F,v \rangle, v ) \rangle |^2 \geq 0.
\end{split}\]
Equality holds if and only if for some $c\neq 0$
\[ \alpha_1 = c\langle \nabla F,v \rangle \text{ and } \oalpha = cv,   \]
which implies
\[ \alpha_1=\langle \nabla F, \oalpha \rangle. \]
But the latter can not happen in $I_-$ according to \eqref{grill}.

The formul{\ae} for $(B^{-1})^T$ and $L$ tell us
\[ B^T L = -\frac{2( \langle \nabla F,\oalpha \rangle - \alpha_1 )}{ 1+|\nabla F|^2 } \nabla^2 F. \]
Notice that $\langle \nabla F,\oalpha \rangle-\alpha_1>0$ on $I_-$ and $-\nabla^2 F$ is positive semi-definite, thus $B^TL$ is symmetric and positive semi-definite. 

Finally, we conclude that 
\[\begin{split}
    j(s,\ox,t) = & 2\xi_1^r\det(A)=2\xi_1^r\det(B+2sC(K+L))\\
    = & 2\xi_1^r\det(B)\det( I+2s B^{-1}C(B^{-1})^T(B^TK+B^TL) ) \\
    = & 2(\langle \nabla F,\oalpha \rangle -\alpha_1 ) \det( I+2s B^{-1}C(B^{-1})^T(B^TK+B^TL)).
\end{split}\]
Notice now that both $B^{-1}C(B^{-1})^T$ and $B^TK+B^TL$ are symmetric and positive definite. Let $Q$ be an invertible matrix such that $QQ^T = B^{-1}C(B^{-1})^T$, then 
\[ B^{-1}C(B^{-1})^T(B^TK+B^TL) = Q \left(Q^T(B^TK+B^TL)Q\right) Q^{-1}.  \]
This shows that the product of $B^{-1}C(B^{-1})^T$ and $B^TK+B^TL$ has positive eigenvalues.
Thus we conclude that 
\[ j(s,\ox,t)\geq 2(\langle \nabla F, \oalpha \rangle -\alpha_1)>0. \]
This completes the proof.
\end{proof}

\section{The grazing set (GS) assumption}\label{gsa}

Most of this section is concerned with providing classes of examples where the grazing set assumption is satisfied for incoming spherical waves.  It becomes more difficult to provide examples as the dimension of the obstacle increases.   Near the end of the section we provide examples for both spherical and planar waves where the grazing set is continuous but not differentiable, so the GS assumption fails to hold.

In view of Remark \ref{d14z}  we are mainly interested in grazing of order $\geq 4$, the cases where it is challenging to understand the nature (possible branching, regularity) of the grazing set because the implicit function theorem cannot be used directly. 

\subsection{Spherical waves grazing 2D obstacles}\label{natureg}

We show first that the GS assumption is always satisfied for incoming spherical waves in 2D.  For plane waves this result is proved in Proposition 8.2 of \cite{ww2023}.

\begin{prop}\label{c1z}
Let $\mathcal{O}$ and $F$ satisfy Assumption \ref{n1z} when $n=2$ and suppose $\phi_i(x,t)=-t+|x-b|$, $b=(1,b_2)$ $(b_2\neq 0)$ satisfies Assumption 
\ref{grill5}.  Then the grazing set assumption, Assumption \ref{gs}, holds with $\zeta=x_2$.

\end{prop}

\begin{proof}
\textbf{1. }Writing $F$ as in \eqref{TF} and noting that $p_j(\oxi^i(0))=0\Leftrightarrow p_j=0$, we see that $\usigma\in G^{2k}_d\setminus G^{2k+1}$ for some $k$, (resp. $\usigma\in G^\infty_d$) if and only if $F$ has the form 
\begin{align}\label{c1y}
\begin{split}
&(a) F(x_2)=1+c_{2k}x_2^{2k}+r(x_2), \;\text{ where }c_{2k}<0\text{ and }r(x_2)=O(|x_2|^{2k+1}),\\
&(b)\text{ respectively, } F(x_2)=1+r(x_2),\;\text{ where }r(x_2)=O(|x_2|^\infty)\leq 0.
\end{split}
\end{align}
We have $(F(x_2),x_2,t)\in G_{\phi_i}\Leftrightarrow F(x_2)-1-F'(x_2)(x_2-b_2)=0$.    In cases (a), (b) this means
\begin{align}\notag
\begin{split}
&(a)\; 0=c_{2k}x_2^{2k}-2kc_{2k}x_2^{2k-1}(x_2-b_2)+O(|x|^{2k})=x^{2k-1}[-2kc_{2k}(x_2-b_2)+O(|x_2|])\\
&(b)\; 0=r(x_2)-r'(x_2)(x_2-b_2).
\end{split}
\end{align}
Clearly, (a) holds for small $x_2$ if and only if $x_2=0$.  In case (b) the strict concavity of $F$ implies
\begin{align}
0\leq r(0)-r(x_2)\leq r'(x_2)(-x_2)\Leftrightarrow 0\leq -r(x_2)\leq r'(x_2)(-x_2).
\end{align}
For $x_2\neq 0$ small this implies $|r'(x_2)|\geq |r(x_2)|/|x_2|$. Since $r'(x_2)\neq 0$ for $x_2\neq 0$ small and $x_2-b_2\sim 1$, we see that  (b) holds for $x_2\neq 0$ small if and only if $x_2=0$.\footnote{Here we use that $r'$ is strictly decreasing, since $F$ is strictly concave.}

\textbf{2. }Condition (b) in the GS assumption holds, since $\urho=(1,0,t,0,-b_2/|b_2|,-1)$ and  $H_p\zeta(\urho)=2(-b_2/|b_2|)\partial_{x_2}(x_2)\neq 0.$

\textbf{3. }Condition (c) in the GS assumption is a consequence of the strict convexity of $\mathcal{O}$ near $(1,0)$.   Indeed, 
points in the graph of $d\phi_0$ near $\usigma$ have the form $(\ox,t',\xi_1^i(\ox)\nabla F(\ox)+\oxi^i(\ox),-1)$ for $(\ox,t')$ near $(0,t)$, so condition (c) follows directly from Remark \ref{d12x}.
\end{proof}

\subsection{Spherical waves grazing nD obstacles with an extra symmetry}\label{nd}

\emph{}


Here we present a class of obstacles $\cO\subset \mathbb{R}^n$, $n\geq 2$,   satisfying  the grazing set assumption.   This class includes examples where the order of grazing is infinite.

\begin{ass}\label{ndA}
Let  $\mathcal{O}\subset \mathbb{R}^n$, $n\geq 2$,  be an obstacle that is strictly convex near $x=(1,0)$,  and which is defined by a function  $F$ as in Assumption \ref{n1z} that satisfies the additional conditions:
\begin{equation}\begin{gathered}\label{r20}
    F(\ox) = 1-h(|\Lambda \ox|^2), \
    h\in C^{\infty}( [0,R);[0,\infty) )\text{ for some }R>0, \\
    h(0)=0, \ h^{\prime}|_{(0,R)}>0, \ h^{\prime\prime}|_{[0,R)}\geq 0,\\ 
    \frac{h(s)}{h'(s)}|_{(0,R)}=sp(s),  \text{ where }p\in C^1([0,R)),\\
    \Lambda \text{ is a nonsingular constant matrix. }
\end{gathered}\end{equation}
\end{ass}

\begin{rem}
\textup{The last condition on $h$ is satisfied with $p\in C^1$ whenever $h^{(k)}(0)\neq 0$ for some $k$.  The condition is also satisfied for some functions that vanish to infinite order at $s=0$, for example, $h(s)=\begin{cases}e^{-s^{-2}},\;s\neq 0\\0,\;s=0\end{cases}$.}

\end{rem}

\begin{prop}\label{rnd}
Suppose $\cO\subset \mathbb{R}^n$, $n\geq 2$ is defined by a function $F$ as in Assumption \ref{ndA}.  
Let $\phi_i=-t+|x-b|, $ where $b=(1,\ob)$, $\ob\in \mathbb{R}^{n-1}\setminus \{0\}$.     The GS assumption, Assumption \ref{gs}, is satisfied if we take 
$\beta=x_1-F(\ox)$,  $\zeta$ as in \eqref{r23z}, and $\usigma=i^*\urho$, where $\urho = (1,0,t, 0, -\frac{\ob}{|\ob|}, -1)$. 
\end{prop}

\begin{proof}
   \textbf{1. } We have $\usigma:=i^*\urho \in \mathcal G_d$, where $\urho = (1,0,t, 0, -\frac{\ob}{|\ob|}, -1)$. 
   We compute
\begin{align}\label{r21}
    \begin{split}
        &(a) \;\nabla F(\ox) =-2h'(|\Lambda\ox|^2)\Lambda^T\Lambda \ox,\\
        &(b)\; \mathcal{H}(\ox):=F(\ox)-1-\nabla F(\ox)\cdot (\ox-\ob)=\\
        &\qquad\quad  -h(|\Lambda\ox|^2)+2h'(|\Lambda\ox|^2)|\Lambda\ox|^2-2h'(|\Lambda\ox|^2)(\Lambda\ox\cdot\Lambda\ob)\\
        &(c)\; \nabla^2 F(\ox)=-2 h'(|\Lambda\ox|^2)\Lambda^T\Lambda - 4 h''(|\Lambda \ox|^2) \ (\Lambda^T\Lambda \ox)\otimes (\Lambda^T\Lambda \ox).
    \end{split}
    \end{align}
\noindent From \eqref{r21}(c) we see that $\nabla^2 F(\ox)<0$ for $\ox\neq 0$.   

   Write the Taylor expansion of $h$ at $s=0$ as
\begin{align*}
h(s)=\sum^k_{j=1}\frac{h^{(j)}(0)}{j!}s^j+O(s^{k+1}),
\end{align*}
and observe that the conditions on $h$ imply that the first nonzero coefficient (if there is one) must be \emph{positive}. 
A computation similar to  \eqref{d14} shows that 
\begin{align}\label{r23}
\begin{split}
\usigma\in G^{2k}_d\setminus G^{2k+1} & \Leftrightarrow h^{(j)}(0)=0 \text{ for }j=1,\dots,k-1\text{ and }h^{(k)}(0)>0;\\
\usigma\in G^{\infty}_d & \Leftrightarrow h^{(j)}(0)=0 \text{ for all }j.
\end{split}
\end{align}
Both cases give $\usigma\in \mathcal G_d$.

\textbf{2. }To verify Assumption \ref{gs} we recall that the grazing set is determined by 
$\mathcal{H}(\ox)=0$.   From \eqref{r21}(b) and the properties of  $h$ we see that this is equivalent to $\zeta(\ox)=0$, where 
\begin{align}\label{r23z}
\zeta(\ox):=\begin{cases}\frac{-h(|\Lambda\ox|^2)}{h'(|\Lambda\ox|^2)}+2|\Lambda\ox|^2-2\Lambda\ox\cdot\Lambda\ob,\;\ox\neq 0\\0,\;\ox=0\end{cases}.
\end{align}
We claim that for $\ox$ small  $\zeta$ is $C^1$ and $\partial_{\ox}\zeta(0)=-2\Lambda^T\Lambda\ob\neq 0.$   
The claim holds if
$f(\ox):=\begin{cases}\frac{h(|\Lambda\ox|^2)}{h'(|\Lambda\ox|^2)},\;\ox\neq 0\\0,\;\ox=0\end{cases}$ satisfies $f\in C^1$ for $\ox$ small with $\nabla f(0)=0$.  That follows easily from the assumptions on $h$.

\textbf{3. }We have $H_p\zeta(\urho)=-2\frac{\ob}{|\ob|}\cdot \partial_{\ox}\zeta(0)=4|\Lambda\ob|^2/|\ob|\neq 0$, so condition (b) of the GS assumption holds.  Condition (c)  is a consequence of strict convexity as in the proof of Proposition \ref{c1z}.  
\end{proof}

\subsection{A class of 3D obstacles for which the GS assumption holds}\label{3do}

The next Theorem is our main positive result on the grazing set  assumption for 3D obstacles and incoming spherical waves.  It shows that for obstacles defined by $x_1<F(\ox)$, the GS assumption is satisfied by incoming spherical waves with phase $\phi_i(x,t)=-t+|x-b|$, where $b=(1,\ob)$, $\ob\in\mathbb{R}^2\setminus\{0\}$, whenever $F(\ox)$ has a Taylor expansion at $0$ of the form \eqref{u1w} with $\nabla^2 G_{2k}(\ox)>0$ for $\ox>0$.  These obstacles are  more general then those considered in the $n-$dimensional result, Proposition \ref{rnd}.


 It will be convenient below sometimes to write $\ox=(x_2,x_3)=(u,v)$.

\begin{theo}\label{u1}
 Suppose for some $k\in\mathbb{N}$ that 
 \begin{align}\label{u1w}
F(\ox)=1-G_{2k}(\ox)+r(\ox),
\end{align}
 where
$G_{2k}$ is a homogeneous real polynomial of degree $2k$ in $\ox=(x_2,x_3)$ such that 
\begin{align}\label{u1ww}
\nabla^2 G_{2k}(\ox)>0 \text{ for }\ox\neq 0,
\end{align}
 and $r(x)=O(|\ox|^{2k+1})$ is a $C^\infty$ remainder term.
Define the obstacle $\mathcal{O}\subset \mathbb{R}^3$ as in Assumption \ref{n1z}  
by $x_1=F(\ox)$, and take the incoming phase 
to be $\phi_i=-t+|x-b|$, $b=(1,\ob)$ with $\ob\in\mathbb{R}^2\setminus \{0\}$.  The GS assumption, Assumption \ref{gs}, is satisfied if we take 
$\beta=x_1-F(\ox)$ and $\usigma=i^*\urho$, where $\urho = (1,0,t, 0, -\frac{\ob}{|\ob|}, -1)$; the function $\zeta$ has the form $x_2-u(x_3)$ for some $C^\infty$ function $u(\cdot)$ or the form $x_3-v(x_2)$ for some $C^\infty$ function $v(\cdot)$.


\end{theo}

\begin{rem}\label{u1x}

\textup{Examples presented in \S\ref{bad} show that even for strictly concave functions $F$ satisfying Assumption \ref{as1},  if the lowest order (nonconstant) homogeneous polynomial $-G_{2k}$  in the Taylor expansion of $F(\ox)$ at $\ox=0$ fails to satisfy \eqref{u1ww}, 
then some or all of the  conclusions of Theorem \ref{u1}  can fail.}

\textup{(a)\;For example, if $b=(1,-1,0)$ and $F(\ox)=1-(x_2^4+x_3^4)$, so $k=2$ and $G_4(\ox)=x_2^4+x_3^4$,  then the grazing set is $C^1$ but not $C^2$; see \S\ref{c1spher}.  In this case $F$ and $-G_4$ are strictly concave, but \eqref{u1ww} does not hold.} 

\textup{(b) If $b=(1,-1,0)$ and  $F(\ox)=1-(x_2^4+x_3^2)$, so $k=1$ and  $G_2(\ox)=x_3^2$, the grazing set has a cusp; see \S\ref{cusp}.   In this case  $F$ is strictly concave, but $-G_2(\ox)$ is not strictly concave.}

\textup{(c) Part of the conclusion of the theorem is that one obtains a $C^\infty$ grazing set for \emph{all} nonzero choices of $b$ in the plane $x_1=1$.   When the condition \eqref{u1ww} fails,  the grazing set can be $C^\infty$ for some choices of $b$ in the plane $x_1=1$ and cusped for other choices.  See Remark \ref{loc}.}

\textup{In these examples $\usigma=i^*(1,0,t,0,-\frac{\ob}{|\ob|},-1)\in G^4_d\setminus G^5$.}
\end{rem}


\begin{proof}[Proof of Theorem \ref{u1}]

\textbf{1.  The case $r=0$. } We first treat the case $r=0$, which occupies most of the proof.
Switching to $\ox=(u,v)$ notation and writing $G=G_{2k}$ in the rest of this proof, we note that the condition 
\begin{align}\label{u3}
\nabla^2G(u,v)>0\text{ for }(u,v)\neq 0
\end{align}
implies $G(u,v)>0$ for $(u,v)\neq 0$.  
Thus, $G(\ob)>0$, so \eqref{d14} shows that $\usigma\in G^{2k}_d\setminus G^{2k+1}$. 

\textbf{2. }The grazing set is defined by the equation
\begin{align}\label{u2}
\mathcal{H}(u,v):=F(u,v)-1-\nabla F(u,v)\cdot [(u,v)-\ob]=0,
\end{align}
or equivalently, since $\nabla G(u,v)\cdot (u,v)=2k G(u,v)$, 
\begin{align}\label{u3a}
 \mathcal{H}(u,v)=(2k-1)G(u,v)+\nabla G(u,v)\cdot (-\ob)=0.
 \end{align}

\textbf{3. }Set $g(u,v):=\nabla G(u,v)\cdot \ob$, a  homogeneous polynomial  of degree $2k-1$.  To understand the zero set of $\mathcal{H}$ near $(0,0)$, we first study the zero set of $g$.  
The homogeneity implies that the real zero set of $g$ is a union of at most $2k-1$ lines through the origin.     We claim  that \eqref{u3}  implies there is only one line.   To see this fix $\eps>0$ small and define the level curve
\begin{align*}
C_\eps:=\{(u,v):G(u,v)=\eps\}.
\end{align*}
This is a compact strictly convex $C^\infty$ curve enclosing $0$ with positive curvature at all points.\footnote{Compactness follows from 
$G(u,v)\geq C|(u,v)|^{2k}$, and the other properties follow from \eqref{u3}.}     Now on $C_\eps$ we have $g(u,v)=0\Leftrightarrow \nabla G(u,v)= a\ob^\perp$ for some $a\neq 0$,  and the positive curvature of $C_\eps$ implies this can happen at only two points of $C_\eps$. Thus, the zero set of $g$ must consist of just one line through the origin.   The line has the form $v=0$ or $u=\alpha v$ for some $\alpha\in\mathbb{R}$.   In the rest of this proof we treat the case $u=\alpha v$; the other case is treated similarly.

\textbf{4. }For $v\neq 0$ divide $g(u,v)$ by $v^{2k-1}$ to obtain 
\begin{align}\label{u4}
g(u,v)=v^{2k-1} \tilde g(z)|_{z=\frac{u}{v}},
\end{align}
where $\tilde g(z)$ is a (real) polynomial in $z$ of degree $2k-1$ such that $\tilde g(\alpha)=0$ for $\alpha$ as in step \textbf{3}.\footnote{When the line in step \textbf{3} has the form $v=0$, divide $g(u,v)$ by $u^{2k-1}$ for $u\neq 0$ to obtain $\tilde g(z)$ and set $z=\frac{v}{u}$.}   We claim $\alpha$ is a \emph{simple} root of $\tilde g$, and this is the key to the proof.  Indeed,  for some smooth polynomial $h$ of degree $2k-2$ we can write
\begin{align}\label{u5}
\begin{split}
&g(u,v)=(u-\alpha v)h(u,v) \Rightarrow \\
&\quad \nabla g(u,v)\cdot \ob=((1,-\alpha)\cdot \ob)\;h(u,v)+(u-\alpha v)\nabla h(u,v).
\end{split}
\end{align}
But $\nabla g(u,v)\cdot \ob=\nabla^2G(u,v)\ob\cdot\ob>0$ for $(u,v)\neq 0$.   Evaluating \eqref{u5} at any point $(u,v)\neq 0$ such that $u=\alpha v$ we can conclude that 
\begin{align}\label{u6}
(1,-\alpha)\cdot \ob\neq 0\text{ and }\;h(\alpha v,v)\neq 0.
\end{align}
 Now $\tilde g(z)=(z-\alpha)\tilde h(z)$ for some polynomial $\tilde h$ of degree $2k-2$, so \eqref{u6} implies $\tilde h(\alpha)\neq 0$.

\textbf{5. Zeros of $\mathcal{H}$. }For $v\neq 0$ setting $z=\frac{u}{v}$ we can rewrite \eqref{u3a} as 
\begin{align}\label{u7}
\begin{split}
&-\mathcal{H}(u,v)=g(u,v)+(1-2k)G(u,v)=v^{2k-1}\left[\tilde g(z)+G^*(u,v;z)\right], 
\end{split}
\end{align}
where $G^*(u,v;z)$ is a polynomial of degree $2k-1$ in $z$, each of whose coefficients is a real multiple of either $u$ or $v$.\footnote{For example, if $k=2$ and $G(u,v)=u^4+u^2v^2+v^4$, then $$-3G(u,v)=v^3(-3uz^3-3uz-3v)=v^3G^*(u,v;z).$$}
Thus, we have
\begin{align}\label{u8}
\tilde g(z)+G^*(u,v;x):=p(u,v;z),
\end{align}
where $p(u,v;z)$ is a polynomial of degree $2k-1$ in $z$, each of whose coefficients $p_i(u,v)$ has the form $c_i+d_iu$ or $e_i+f_iv$ for some real constants $c_i, d_i, e_i, f_i$; in particular,
\begin{align}\label{u9}
p(0,0;z)=\tilde g(z).
\end{align}
Now $p(0,0;\alpha)=0$ and by step \textbf{4} $\partial_zp(0,0;\alpha)\neq 0$.   Thus, the implicit function theorem implies 
that there is a $C^\infty$ function $z(u,v)$ defined near $(u,v)=0$ with $z(0,0)=\alpha$ such that 
\begin{align}\label{u10}
p(u,v;z)=0\Leftrightarrow z=z(u,v).
\end{align}
With \eqref{u7} this implies that for $(u,v)$ near $0$, 
\begin{align}\label{u11}
 \mathcal{H}(u,v)=0\Leftrightarrow p(u,v;z(u,v))=0.
\end{align}
Thus, the zero set of $\mathcal{H}$ near $(u,v)=0$ is given by $\zeta^*(u,v)=0$, where $\zeta^*(u,v)=u-z(u,v)v$. 
Since $\partial_u\zeta^*(0,0)\neq 0$, another application of the implicit function theorem shows that this zero set is given by $\zeta(u,v)=0$ where 
\begin{align}
\zeta(u,v)=u-u(v)
\end{align}
for some $C^\infty$ function $u(v)$ such that $u(0)=0$.

\textbf{6. }Conditions (b) and (c) of the GS assumption are both consequences of strict convexity. 
We have $H_p=2\xi_1\partial_{x_1}+2\overline\xi\partial_{\ox}-2\tau\partial_t$, so with $\urho=(1,0,t,0,-\ob,-1)$ we obtain
\begin{align*}
H_p\zeta (\urho)=-2\ob\cdot \partial_{\ox}\zeta(0)=-2\ob\cdot (1,-\alpha)\neq 0
\end{align*}
by \eqref{u6}. This gives (b).  Condition (c)  is a consequence of strict convexity as in the proof of Proposition \ref{c1z}.  

\textbf{7.  The general case $r(\ox)\neq 0$. } The above argument goes through almost without change.  In place of \eqref{u3a} we now have 
\begin{align}\label{u3z}
 \begin{split}
 &\mathcal{H}(u,v)=(2k-1)G(u,v)+\nabla G(u,v)\cdot (-\ob)+\left[r(u,v)-\nabla r(u,v)\cdot ((u,v)-\ob)\right]=0.
\end{split}
 \end{align}
The term in brackets is a sum of terms of the form $b_\alpha(u,v)(u,v)^\alpha$ where $\alpha=(\alpha_u,\alpha_v)$ is a multi-index of order $2k$ and $b_\alpha(u,v)$ is $C^\infty$.   We define $p(u,v;z)$ as in \eqref{u8}, where 
$p(u,v;z)$ is a polynomial of degree $2k-1$ in $z$, each of whose coefficients $p_i(u,v)$ now has the form $c_i+d_i(u,v)u$ or $e_i+f_i(u,v)v$ for some constants $c_i, e_i$ and $C^\infty$ functions $d_i(u,v)$, $f_i(u,v)$.   The polynomial $\tilde g$ is unchanged, so the application of the implicit function theorem works just as before, as does the rest of the proof.


\end{proof}

\begin{rem}
\textup{Proposition 8.4 of \cite{ww2023} showed by a different argument that for incoming plane waves associated to phases
$\phi_i(x,t)=-t+\ox\cdot\otheta$, $\otheta\in S^{n-2}$ and obstacles $\mathcal{O}$ satisfying the conditions of Theorem \ref{u1}, the GS assumption always holds with a grazing set that is at least $C^1$.  The proof of Theorem \ref{u1} implies that the grazing set is actually $C^\infty$.  To obtain this conclusion recall that for incoming plane waves the function $\mathcal{H}$ as in \eqref{u3a} that defines $\pi_{\ox}G_{\phi_i}$ is 
$\nabla G(u,v)\cdot \otheta$.   One can then repeat steps \textbf{3}-\textbf{7} of the above proof for this new choice of $g(u,v)$ with the minor change that the  function $G^*(u,v;z)$ as in \eqref{u8} is now completely determined by the higher order term $r(\ox)$ in expansion of $F(\ox)$.}

\end{rem}

\subsection{Branching does not occur in 3D near an isolated zero of Gauss curvature}\label{branch}

In this section we show that for a much more general class of 3D obstacles than was considered in \S\ref{3do}, \emph{branching} does not occur.  That is, 
$\pi_xG_{\phi_i}$ is at least a single {continuous} curve passing through $x=(1,0)$.   {Branching} occurs when the set $\pi_{\ox}G_{\phi_i}\setminus \{0\}$ consists of more than two smooth curves that meet and terminate at $\ox=0$. 
Observe that each curve in $\pi_{\ox}G_{\phi_i}\setminus \{0\}$ lifts under the map $\ox\mapsto (F(\ox),\ox)$ to a curve in $\pi_xG_{\phi_i}$.   We now consider strictly convex obstacles for which the condition \ref{u1ww} in Theorem \ref{u1} need not hold.  Examples below show that the GS assumption can fail for such obstacles.

\subsubsection{No branching}
We write $x=(x_1,\overline{x})$, where $\ox=(x_2,x_3)$.



\begin{ass}\label{as1}
We take $\mathcal{O}\subset \mathbb{R}^3$ and $F$ as in Assumption \ref{n1z} and require that  
\begin{align}\label{c1}
\begin{split}
&\nabla^2F(\ox)<0\text{ for }\ox\neq 0 \text{ when }\ox\text{ is near }0.
\end{split}
\end{align}
 \end{ass}

 \begin{rem}\label{curv}
\textup{Assumption \ref{as1} is equivalent to the statement that the boundary of $\mathcal{O}$ is strictly convex with $x=(F(0),0)=(1,0)$ being an isolated zero of the Gauss curvature.}
 
 \end{rem}


 We  take a spherical incoming phase  $\phi_i=-t+|x-b|$, where 
$b=(1,\ob)$  with $\ob\neq 0$.  
 The grazing set near  $x=(1,0)$ is 
$$G_{\phi_i}=\{(F(\ox),\ox,t) \mid F(\ox)-1-\nabla F(\ox)\cdot (\ox-\ob)=0, \ \ox \text{ near }0, \ t\in\mathbb{R}\}.$$
We claim that for $F$ as above,  the set 
$$\pi_{\ox}G_{\phi_i}:=\{\ox \text{ near }0 \mid \mathcal{H}(\ox):=F(\ox)-1-\nabla F(\ox)\cdot (\ox-\ob)=0\}$$ consists of a single {continuous} curve through $\ox=0$.



\begin{theo}[No branching]\label{A}
Let $F$   
satisfy Assumption \ref{as1} and take $\phi_i(x,t)=-t+|x-b|$, where $b=(1,\ob)$ with $\ob\neq 0$. Near $\ox=0$ the set $\pi_{\ox}G_{\phi_i}\setminus \{0\}$ consists of two $C^\infty$ curves, one in $x_3>0$ and one in $x_3<0$.  The set $\pi_{\ox}G_{\phi_i}$ is  near $\ox=0$ a single \emph{continuous} curve through $\ox=0$. 

\end{theo}



\begin{proof}

\textbf{1. Smoothness of $\pi_{\ox}G_{\phi_i}$ away from $\ox=0$.}  For $\ox$ near $0$, $\ox\neq 0$ we compute
\begin{align}\label{c2}
\nabla\mathcal{H}(\ox)=\nabla F(\ox)-\nabla^2 F(\ox)(\ox-\ob)-\nabla F(\ox)=-\nabla^2 F(\ox)(\ox-\ob)\neq 0,
\end{align}
since $\nabla^2 F<0$ for $\ox\neq 0$.   Thus, $\pi_{\ox}G_{\phi_i}$ is $C^\infty$ away from $\ox=0$.

In steps \textbf{2}-\textbf{7} we first consider the case $b=(1,-1,0)$.  

\textbf{2. The  curve $C(A^*)$. }In the $x_1,x_2$ plane fix a point $A^*=(F(x_2^*,0),x_2^*)$  near $(x_1,x_2)=(1,0)$ with $x_2^*<0$ on the  strictly concave curve $x_1=F(x_2,0)$, and draw the infinite ray from  $(1,-1)$ toward $A^*$.  
This ray intersects $x_1=F(x_2,0)$ at $A^*$ and at a second point $B^*=(F(x_2^{**},0),x_2^{**})$ in $x_2>0$. 
Let $P(A^*)$ denote the half-plane in $x-$space consisting of all points that project to this ray under $(x_1,x_2,x_3)\to (x_1,x_2)$.   We have
\begin{align}\label{c3}
P(A^*)=\{(x_1,x_2,x_3):(x_1-1)(x_2^*+1)+(x_2+1)(1-F(x_2^*,0))=0\}.
\end{align}
Recall that the boundary of the obstacle $\mathcal{O}$ consists of points of the form $(F(\ox),\ox)$.  
Set 
\begin{align}\label{c4}
\begin{split}
&C(A^*)=P(A^*)\cap \partial\mathcal{O}=\\
&\qquad\quad \{(F(\ox),\ox): K(\ox):=(F(\ox)-1)(x_2^*+1)+(x_2+1)(1-F(x_2^*,0))=0\}.
\end{split}
\end{align}
We will show that $C(A^*)$ is a $C^\infty$ closed strictly convex curve that contains exactly two points in $\pi_x G_{\phi_i}$.

\textbf{3. Smoothness of $C(A^*)$. }We show that near $\ox=0$
\begin{align}\label{c5a}
\nabla K(\ox)=(F_{x_2}(\ox)(x_2^*+1)+1-F(x_2^*,0), F_{x_3}(\ox)(x_2^*+1))
\end{align}
is nonzero at any $\ox$ where $K(\ox)=0$. If $F_{x_3}(\ox)\neq 0$, then $\nabla K(\ox)\neq 0$, so suppose $F_{x_3}(\ox)=0$. 
We show now that the conditions 
\begin{align}\label{c5}
(a) F_{x_3}(\ox)=0, \;(b) K(\ox)=0, \;(c) K_{x_2}(\ox)=0
\end{align}
are incompatible.     Conditions \ref{c5} (b),(c) imply
\begin{align}\label{c6}
\begin{split}
&F_{x_2}(\ox)=\frac{F(x_2^*,0)-1}{x_2^*+1}=\frac{F(\ox)-1}{x_2+1}<0.
\end{split}
\end{align}
 A Taylor expansion gives
\begin{align}\label{c7}
\begin{split}
1-F(\ox)= & F(0)-F(\ox) \\
= & -\nabla F(\ox)\ox+\int^1_0(1-s)\langle\nabla^2F(\ox(1-s))ds \;\ox,\ox\rangle\\
= & -F_{x_2}(\ox)x_2+\int^1_0(1-s)\langle\nabla^2F(\ox(1-s))ds \;\ox,\ox\rangle,
\end{split}
\end{align}
where we used \eqref{c5}(a) to get the third equality.  From \eqref{c6} and \eqref{c7} we obtain
\begin{align}
\begin{split}
F_{x_2}(\ox)(x_2+1)= & F_{x_2}(\ox)x_2-\int^1_0(1-s)\langle\nabla^2F(\ox(1-s))ds \;\ox,\ox\rangle\\
\Rightarrow \ F_{x_2}(\ox)= & -\int^1_0(1-s)\langle\nabla^2F(\ox(1-s))ds \;\ox,\ox\rangle>0,
\end{split}
\end{align}
which contradicts \eqref{c6}.

\textbf{4. Nonvanishing curvature of  $C(A^*)$.} Let $Q=(F(\ox),\ox)\in C(A^*)$ be a point where $K_{x_2}(\ox)\neq 0$.\footnote{Points where $K_{x_3}(\ox)\neq 0$ are treated similarly.}   Near $\ox$ the curve $
\pi_{\ox}C(A^*)$ is the graph of a $C^\infty$ function $x_2=x_2(x_3).$ 
Differentiating $K(x_2(x_3),x_3)=0$ twice with respect to $x_3$ we obtain
\begin{align}\label{c8}
\begin{split}
&K_{x_2}x_2'+K_{x_3}=0,\\
&\langle\nabla^2K(x_2(x_3),x_3)(x_2',1), (x_2',1)\rangle +K_{x_2} x_2''=0\Rightarrow x_2'' \neq 0.
\end{split}
\end{align}
Here we used $\nabla^2K(\ox)=\nabla^2F(\ox)(x_2^*+1)<0.$  Near $Q$ we can parametrize $C(A^*)$ by 
\begin{align}\label{c9}
r(x_3)=(F(x_2(x_3),x_3),x_2(x_3),x_3),
\end{align}
with which we compute
\begin{align}\label{c10}
\begin{split}
r'(x_3)= & (F_{x_2}x_2'+F_{x_3},x_2',1):=(f,x_2',1)\\
r''(x_3):= & (g,x_2'',0),
\end{split}
\end{align}
so $r'\times r''=(-x_2'',g,*)\neq  0$ by \eqref{c8}.\footnote{We have 
$F_{x_2}(x_2^*,0)=\int^1_0 F_{x_2x_2}(s(x_2^*,0))ds\cdot x^*_2 >0$,  so $K_{x_2}(x_2^*,0)>0$ by \eqref{c5a}.   This implies $x_2''>0$ at $(x_2^*,0)$ by \eqref{c8}.}
    Now apply the curvature formula for a parametrized curve\begin{align}\label{c11}
\kappa(x_3)=\frac{|r'(x_3)\times r''(x_3)|}{|r'(x_3)|^3}\neq 0.
\end{align}

\textbf{5. Grazing points on $C(A^*)$.} We now show that $C(A^*)$ contains one grazing point in $x_3>0$ and one in $x_3<0$.  The directed infinite ray from $(1,-1)$ toward $A^*$  described in step \textbf{2}, ray($A^*)$, makes an angle of $\theta=0$ with the $x_1x_2$ plane. Keeping $b=(1,-1,0)$ fixed, we now increase the angle $\theta$ that ray($A^*)$ makes with the $x_1x_2$ plane, always keeping the rotated ray, ray($\theta$), in the plane $P(A^*)$.  As $\theta$ increases let $A(\theta)$ denote the point closer to $b$ at which ray($\theta$) intersects $C(A^*)$.  Strict convexity of $C(A^*)$ implies that ray($\theta)\cap C(A^*)$ contains either two points, one point, or zero points.

  Denote by $\alpha(\theta)$ the counterclockwise angle, viewed from the positive $x_1$ axis, that the tangent vector to $C(A^*)$ at 
$A(\theta)$ makes with ray($A^*)=$ray$(0)$.\footnote{At $A(0)$ take the tangent vector to $C(A^*)$ with positive $x_3-$component;  recall \eqref{c10}.  As  $\theta$ increases from $0$  choose the tangent vectors continuously.}
We have $\alpha(0)>0$ and by strict convexity of $C(A^*)$, $\alpha(\theta)$ strictly decreases as $\theta$ increases from $0$.  
The strict monotonicity of $\alpha(\theta)$  implies that if there is a positive angle $\theta$, call it $\theta_g$,  where $\theta_g=\alpha(\theta_g)$, then $\theta_g$ is unique.  The point $A(\theta_g)$ is then the unique grazing point on $C(A^*)$ in $x_3>0$.  

To see that $\theta_g$ exists, define
$$\theta_g=\sup\{\theta\in [0,2\pi): \text{ray}(\theta)\cap C(A^*) \text{ contains exactly two points}\}.$$
Indeed, the sup exists since $\pi$ is an upper bound and ray$(0)\cap C(A^*)$ (resp. ray$(\pi)\cap C(A^*)$)  contains exactly two (resp. zero) points. 
Moreover, if ray$(\theta_g)\cap C(A^*)$ contains two points or zero points, then by continuity $\theta_g$ can't be the sup.  
Thus, ray$(\theta_g)\cap C(A^*)$ contains exactly one point, so $\theta_g=\alpha(\theta_g)$.

\textbf{6. No branching. }With \eqref{c4} we see that $P(A^*)$ contains exactly two points of $\pi_xG_{\phi_i}$, one in $x_3>0$ and one in $x_3<0$.  
Moreover,
$$\cup_{A^*\text{ near }(1,0), \;x_2^*\leq 0}\;\left(P(A^*)\cap \pi_x G_{\phi_i}\right)$$
is equal to the set of points in $\pi_xG_{\phi_i}$ near $(1,0,0)$.  Step \textbf{1} implies that as $x_2^*\nearrow 0$, the points
of $P(A^*)\cap \pi_x G_{\phi_i}$ trace out two curves that are $C^\infty$ away from $(1,0,0)$.  The curves approach $(1,0,0)$ continuously as $x_2^*\nearrow 0$, since $P(A^*)$ approaches the plane $x_1=0$.

\textbf{7. General $\ob$. } First observe that the above argument applies equally well when $\ob=(-1,0)$ is replaced by $\ob=(a,0)$, $a<0$.  Suppose now that $b=(1,\ob)$ with $\ob\neq 0$ in $\phi_i=-t+|x-b|$.  
We can rotate $\ox$ coordinates keeping the $x_1$-coordinate unchanged, so that the rotated $\ob$ equals $(a,0)$ for some $a<0$.     Assumption \ref{as1} continues to hold in the new coordinates, so we can apply the above argument to rule out branching for any choice of $\ob\neq 0$.
\end{proof}

\begin{rem}[Plane waves]
\textup{The no branching result holds for incoming plane waves by a much simpler argument.  Consider the incoming phase
$\phi_i(x,t)=-t+\ox\cdot\otheta$, where $|\otheta|=1$.   The Gauss map $n:\partial\mathcal{O}\to S^2$ is the map given by 
$$(F(\ox),\ox)\to  (1,-\nabla F(\ox))/|(1,-\nabla F(\ox))|.$$  Let $C_{\otheta}\subset S^2$ be the great circle consisting  of vectors orthogonal to $(0,\otheta)$ and note that $n(1,0)=(1,0)\in C_{\otheta}$.    The grazing set $\pi_xG_{\phi_i}$ is the preimage $n^{-1}(C_{\otheta})$.   Assumption \ref{as1} implies that near $(1,0)$, the map $n$ is one-one and continuous with a continuous inverse, so $\pi_xG_{\phi_i}$ is a continuous curve near $(1,0)$.\footnote{Thanks to Mohammad Ghobi  for this observation.}}

\end{rem}

\subsection{Examples where the grazing set assumption fails or just barely holds}\label{bad}

We now present examples showing that there exist smooth strictly convex obstacles $\mathcal{O}\subset \mathbb{R}^3$ satisfying Assumption \ref{n1z}  for which the GS assumption fails for incoming spherical waves; the grazing set  is continuous but not differentiable at a high-order grazing point. It has a cusp.  In another example the GS assumption holds, but just barely; it is $C^1$ but not $C^2$.

We  also give examples for incoming planar waves where the GS assumption fails or holds just barely.\footnote{At the time of writing \cite{ww2023} we did not have such examples for incoming planar waves.}

\subsubsection{Cusped grazing sets: examples where the grazing curve is continuous but not differentiable at $v=0$; spherical waves.}\label{cusp}
The simplest example is probably $F(u,v)=1-G(u,v)$, where $G(u,v)=u^4+v^2$.   In this case the obstacle $\mathcal{O}$ defined as in Assumption \ref{n1z} is strictly convex near $x=(1,0)$,   
but we don't have $\nabla^2 G(u,v)>0$ for $(u,v)\neq 0$.   
Again we take the incoming phase to be $\phi_i=-t+|x-b|$, where $b=(1,\ob)$, $\ob=(-1,0)$.

With $\urho=(1,0,t,0,-\ob,-1)$ we have $\usigma=i^*\urho\in G^4_d\setminus G^5.$ The equation \eqref{u2} is 
\begin{align*}
3u^4+4u^3+v^2=0,
\end{align*}
which is readily seen to define $u$ as a function of $v$.
The equation implies 
\begin{align}\label{a5}
u=-v^{2/3}(4+3u)^{-1/3}\Rightarrow |u|\sim v^{2/3}
\end{align}
  for $(u,v)$ small.  
Expanding we obtain 
\begin{align*}
u=-4^{-1/3}v^{2/3}+O(v^{4/3}),
\end{align*}
which implies $u'(0)$ does not exist.   See Figure \ref{f-1}.


\begin{figure}[t] 
   \centering
   \includegraphics[width=0.3\textwidth]{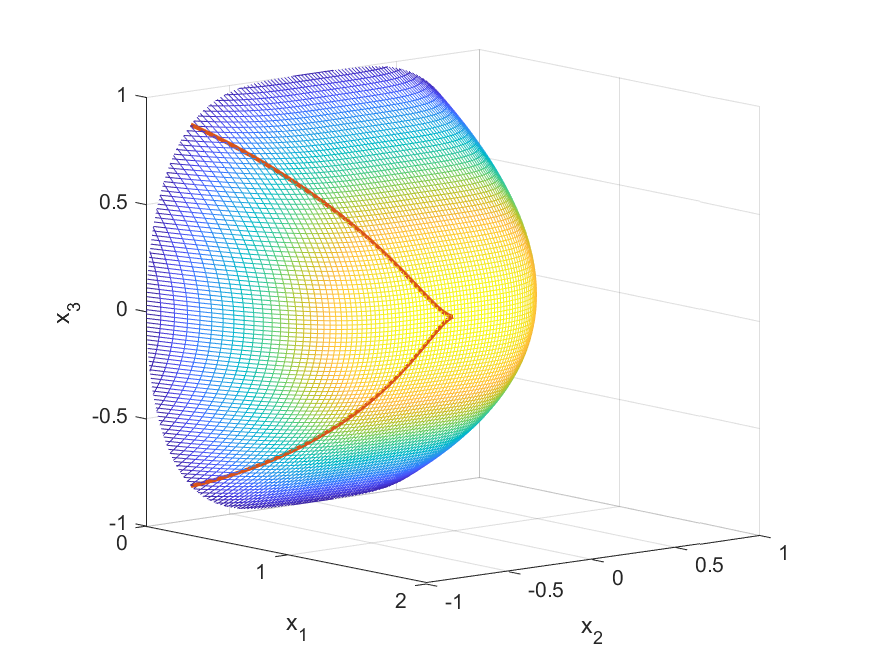} 
   \includegraphics[width=0.3\textwidth]{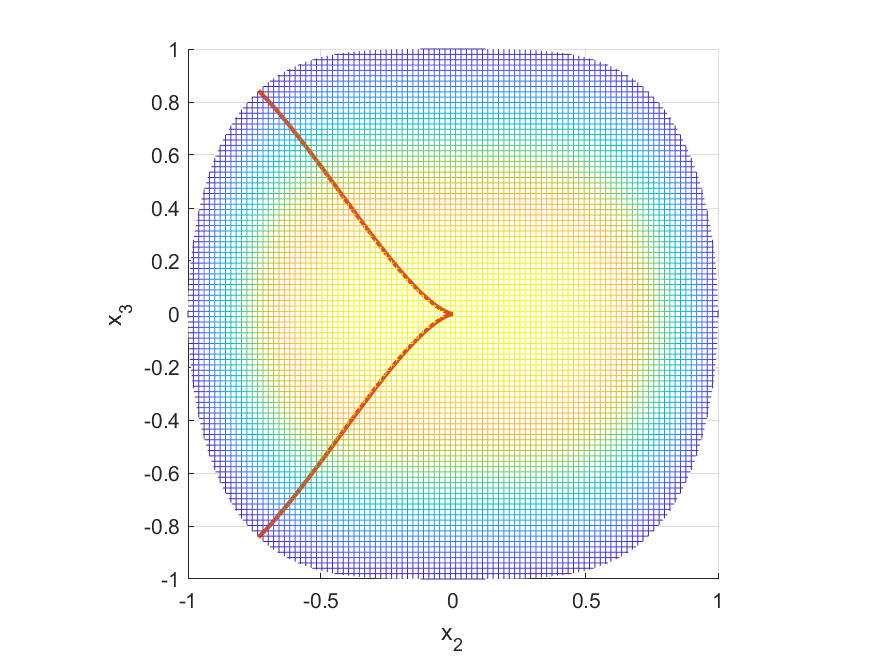} 
   \includegraphics[width=0.3\textwidth]{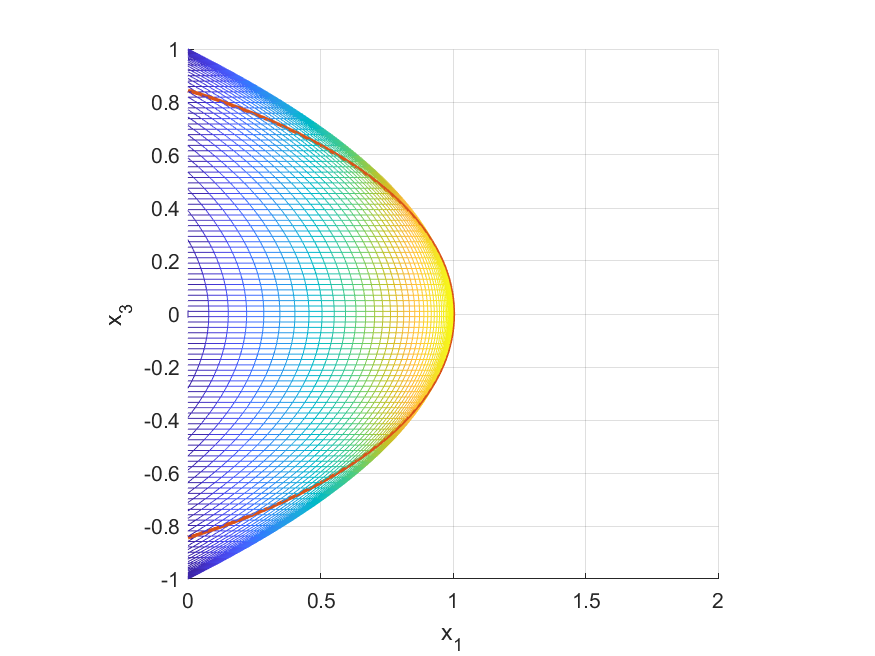}\\
   \includegraphics[width=0.3\textwidth]{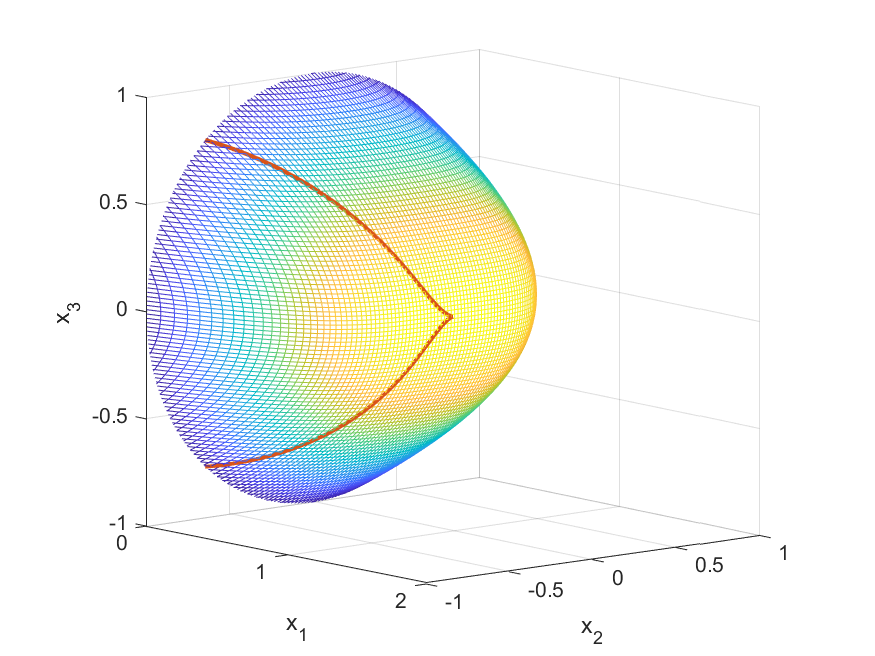}
   \includegraphics[width=0.3\textwidth]{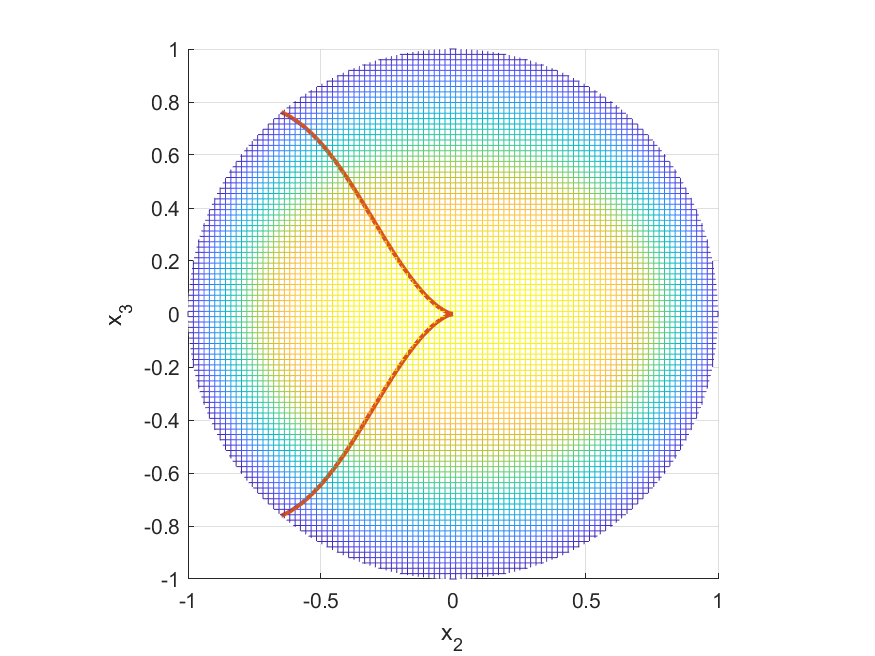} 
   \includegraphics[width=0.3\textwidth]{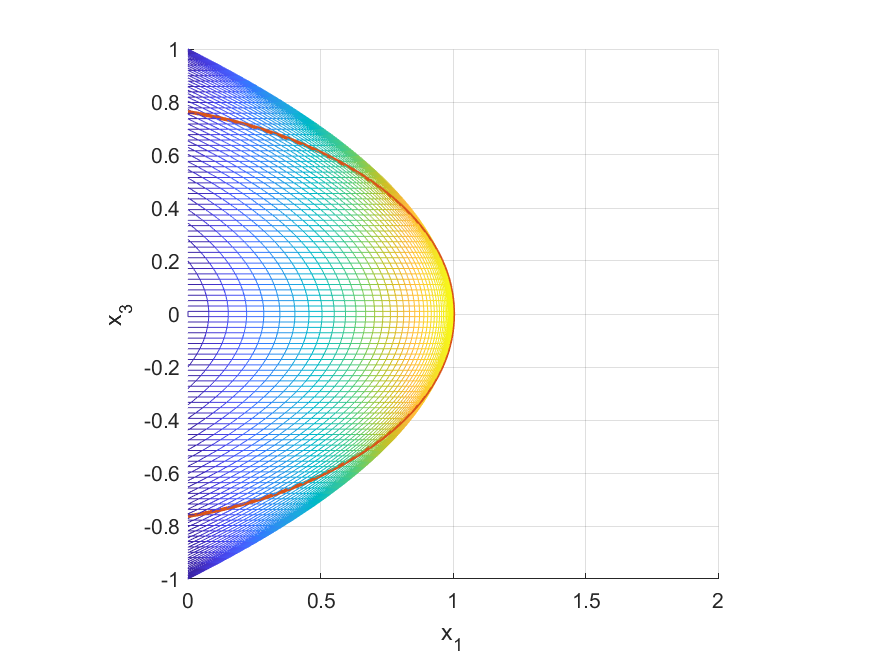} 
   \caption{Light Source at $(1,-1,0)$. First row: Cusped grazing set for the obstacle given by $x_1=1-(x_2^4+x_3^2)$. Second row: Cusped grazing set for the obstacle given by $x_1=1-(x_2^4+x_2^2 x_3^2+x_3^2)$. Colored according to level sets $x_1=\mathrm{const.}$.}
   \label{f-1}
\end{figure}


Here is another example of a cusped grazing set, but now $\nabla^2 G(u,v)>0$ for $(u,v)\neq 0$, 
Take $F(u,v)=1-G(u,v)$, where $G(u,v)=u^4+u^2v^2+v^2$.   The function $F$ satisfies Assumption \ref{as1}.   For $\urho$ as above we again have $\usigma=i^*\urho\in G^4_d\setminus G^5$.
Now the equation \eqref{u2} is 
\begin{align}\label{a6}
\begin{split}
&\mathcal{H}(u,v)=3(u^4+u^2v^2)+4u^3+2uv^2+v^2=0 \\
&\Leftrightarrow  u^3(4+3u)=-v^2(1+2u+3u^2),
\end{split}
\end{align}
whic implies that points in the zero set of $\mathcal{H}$ satisfy\footnote{To see that \eqref{a6} defines a \emph{function} $u=u(v)$ near $(u,v)=0$, suppose 
$0=\mathcal{H}(u_1,v)-\mathcal{H}(u_2,v)$, where $u_1>u_2$.  Factor out $u_1-u_2>0$ and use \eqref{a12} to deduce that the second factor is also positive, contradiction.}
\begin{align}\label{a12}
u=-4^{-1/3}v^{2/3}+O(v^{4/3}).
\end{align}

From \eqref{a12} it is clear that the limit defining $u'(0)$ does not exist.   We can get more information if we 
rewrite \eqref{a6} as
\begin{align}\label{a10}
u[(4u^2+2v^2)+3(u^3+uv^2]:=uB(u,v)=-v^2,
\end{align}
and note that $B(u,v)\gtrsim u^2+v^2\sim v^{4/3}$ for $(u,v)$ small.  Now differentiate \eqref{a10} with respect to $v$, divide by $B$,  and use 
$|v/u^2|\sim v^{-1/3}$ to obtain $|u'(v)|\sim |v|^{-1/3}$ for $v\neq 0$.   One can analyze $u''(v)$  in a similar way to see that $0<u''\sim v^{-4/3}$ for $v\neq 0$, so the graph of $u=u(v)$ is concave up viewed from the positive $u-$axis.  Thus, $u(v)$ ``behaves" near $v=0$ just like the function $w(v)=v^{2/3}$.  See Figure \ref{f-1}.

\begin{rem}\label{explain}
\textup{The apparent smoothness of the projection $\pi_{x_1,x_3}G_{\phi_i}$ in the second column Figure \ref{f-1} might seem surprising.
This is simply explained using \eqref{a5} by  observing that this projection is the graph of 
\begin{align}
\begin{split}
x_1= & F(x_2(x_3),x_3)=1-\left[x_3^{8/3}(4+3x_2(x_3))^{-4/3}+x_3^2\right]\\
=& 1-x_3^2\left[1+x_3^{2/3}(4+3x_2(x_3))^{-4/3}\right]\approx 1-x_3^2.
\end{split}
\end{align}
A similar explanation accounts for the regularity of $\pi_{x_1,x_3}G_{\phi_i}$ in the remaining figures.}
\end{rem}

\emph{Effect of changing source location on regularity of the grazing set.}
It is clear that the grazing set depends on source location.  Here we show that the regularity of the grazing set can also depend on source location.

   Again take $F=1-G(u,v)$ where $G(u,v)=u^4+v^2$, but now put the source at $c=(1,0,1):=(1,\overline{c})$, so the incoming phase is $\phi_i=-t+|x-c|$ and $$\sigma^*=i^*(1,0,t,0,-\overline{c},-1)\in G^4_d\setminus G^5.$$
The equation \eqref{u2} is now
   \begin{align*}
   3u^4+v^2-2v=0,
   \end{align*}
   which implies the grazing set is given near $(u,v)=0$ by $v=2-\sqrt{4-12u^4}$, a $C^\infty$ function.  Recall that when $b=(1,-1,0)$ we get a cusped grazing set for this obstacle.\footnote{The grazing set for $b=(1,1,0)$ has a cusp similar to that for $b=(1,-1,0)$.}
   
   \begin{rem}\label{loc}
   \textup{This dependence of the regularity of the grazing set on source location in the plane $x_1=1$ reflects the failure of 
    condition \eqref{u1ww} in Theorem \ref{u1} to hold.   Indeed, while $F(u,v)=1-(u^4+v^2)$ is strictly concave, the function $(u,v)\to -v^2$ is not.} 
   
   
   \end{rem}

   \subsubsection{A grazing set that is $C^1$ but not twice differentiable for incoming spherical waves}\label{c1spher}
Take $F(u,v)=1-G(u,v)$, where $G(u,v)=u^4+v^4$.     The incoming phase is $\phi_i=-t+|x-b|$, where $b=(1,-1,0)$.   Here again $\usigma=i^*(1,0,t,0,-\ob,-1)\in G^4_d\setminus G^5$.  In this case $-G$ is strictly concave but the assumption \eqref{u1ww} of Theorem \ref{u1} fails; we only have $-\nabla^2G\geq 0$ for $(u,v)\neq 0$.
The grazing curve is now 
\begin{align}\label{a10z}
\begin{split}
&\mathcal{H}(u,v)=3(u^4+v^4)+4u^3=0 \\
&\Rightarrow u=-3^{1/3}v^{4/3}(4+3u)^{-1/3}\Rightarrow u\sim v^{4/3}\text{ for }(u,v)\text{ near }0.
\end{split}
\end{align}
One checks as in \S\ref{cusp} that this equation determines $u$ as a function of $v$ near $(u,v)=0$.   The second line of  \eqref{a10z} implies $u'(0)=0$ and the implicit function theorem shows that in $v\neq 0$ the function $u(v)$ is $C^1$,
so now we have $u(v)$ is differentiable near $v=0$.



To see that $u$ is $C^1$ near $v=0$ rewrite $\mathcal{H}(u,v)=0$ as $u(4u^2+3u^3)=-3v^4$ and differentiate with respect to $v$ obtaining
\begin{align}\label{a9z}
u'(4u^2+3u^3)+u(8uu'+9u^2u')=u'(12u^2+12u^3)=-12v^3.
\end{align}
For $v\neq 0$ small divide by $12u^2+12u^3$ and use \eqref{a10z} to see that 
\begin{align}\label{a11z}
|u'(v)|\sim |v^3/u^2|\sim v^{1/3}. 
\end{align}
Thus $u$ is $C^1$ near $0$.  

From \eqref{a11z} it is clear that the limit defining $u''(0)$ does not exist.  We can get more information by differentiating \eqref{a9z} with respect to $v$ to obtain
\begin{align*}
|u''(v)|\sim v^{-2/3} \text{ for }v\neq 0.
\end{align*}

\begin{rem}\label{allw}
\textup{In the above example the function $g(u,v)$ as in \eqref{u4} is 
\begin{align}
g(u,v)=4u^3=v^3\tilde g(z)|_{z=\frac{u}{v}}, \; \text{ where }\tilde g(z)=4z^3,
\end{align}
so $\tilde g$ now has a triple root.   This and the failure of $\zeta=u-u(v)$ to be $C^\infty$ (or even $C^2$) reflect the fact that 
$\nabla^2(u^4+v^4)$ fails to be positive definite for $(u,v)\neq 0$; it is only positive semidefinite for $(u,v)\neq 0$.}

\end{rem}

To conclude  \S \ref{gsa} we return for a moment to the case of incoming plane waves to provide a couple of examples we had not found at the time of writing of \cite{ww2023}.

\subsubsection{A cusped grazing set for incoming plane waves}
 Take $F(u,v)=1-G(u,v)$, where $G(u,v)=u^4+2u^2v^2+uv^2+v^2$.     The incoming phase is $\phi_i=-t+(u,v)\cdot \otheta$, where $\otheta=(1,0)$.   The point $$\usigma=i^*(1,0,t,0,\otheta,-1)\in G^4_d\setminus G^5$$
and $\nabla^2 G(u,v)>0$ for $(u,v)\neq 0$.   When the source is at infinity the grazing curve is defined by 
$\nabla G(u,v)\cdot \otheta=G_u=0$ instead of \eqref{u2}.  In this case we obtain
\begin{align*}
v^2+4uv^2+4u^3=0; \text{ that is }4u^3=-v^2(1+4u).
\end{align*}
This is similar to \eqref{a6}. The same analysis shows we have a cusped grazing curve.  See Figure \ref{f-3}.

\begin{figure}[t] 
   \centering
   \includegraphics[width=0.3\textwidth]{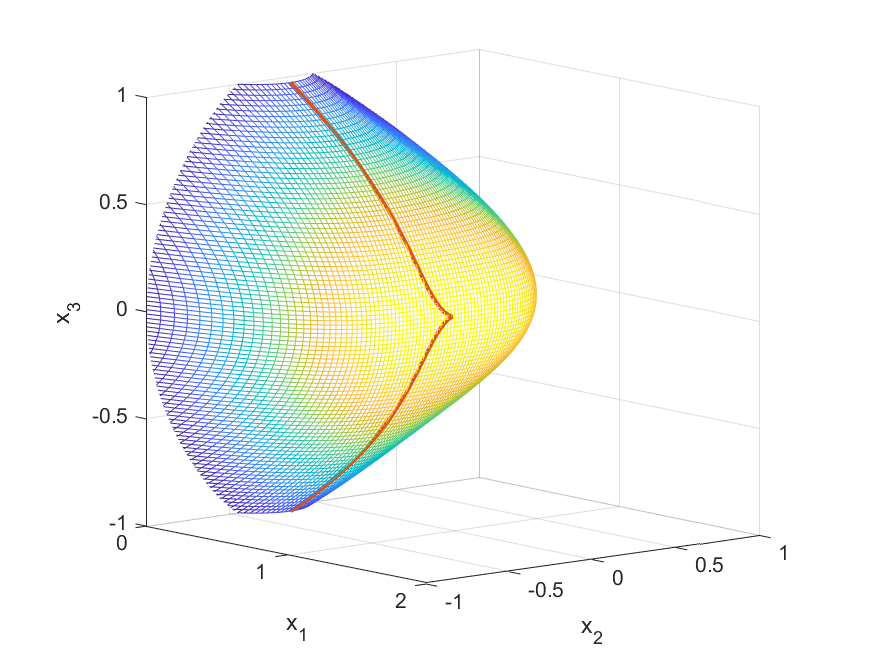} 
   \includegraphics[width=0.3\textwidth]{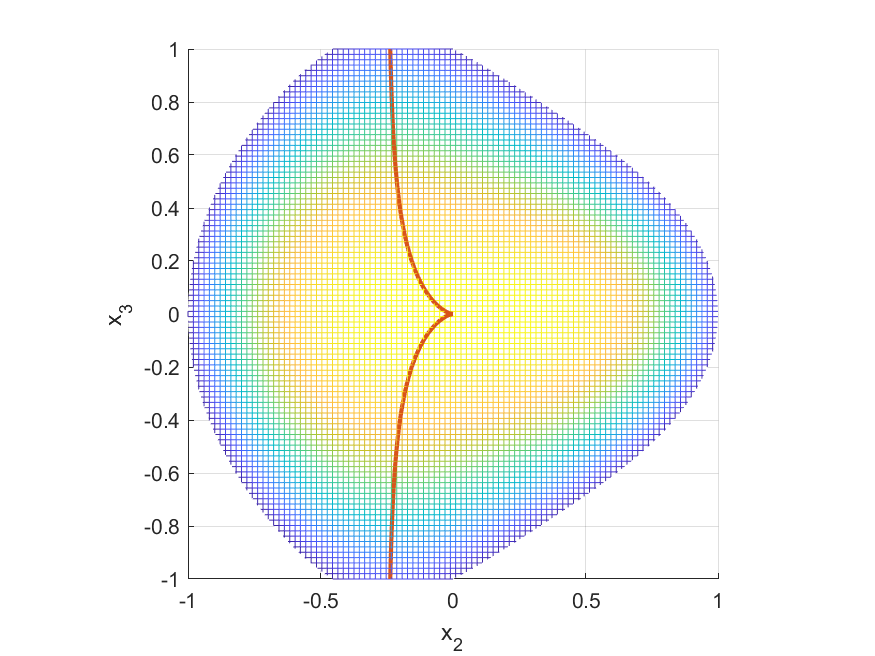} 
   \includegraphics[width=0.3\textwidth]{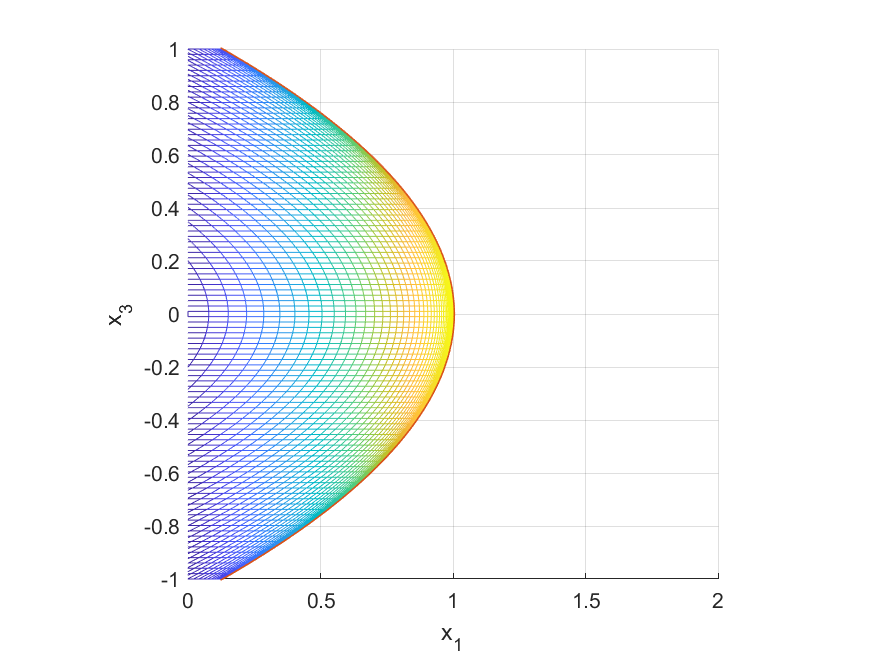}\\
   \includegraphics[width=0.3\textwidth]{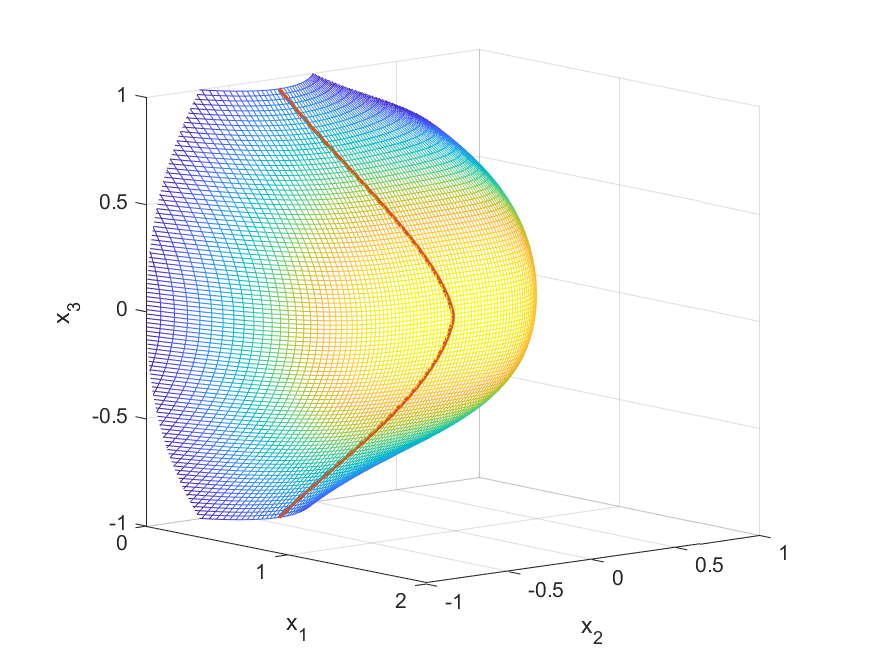} 
   \includegraphics[width=0.3\textwidth]{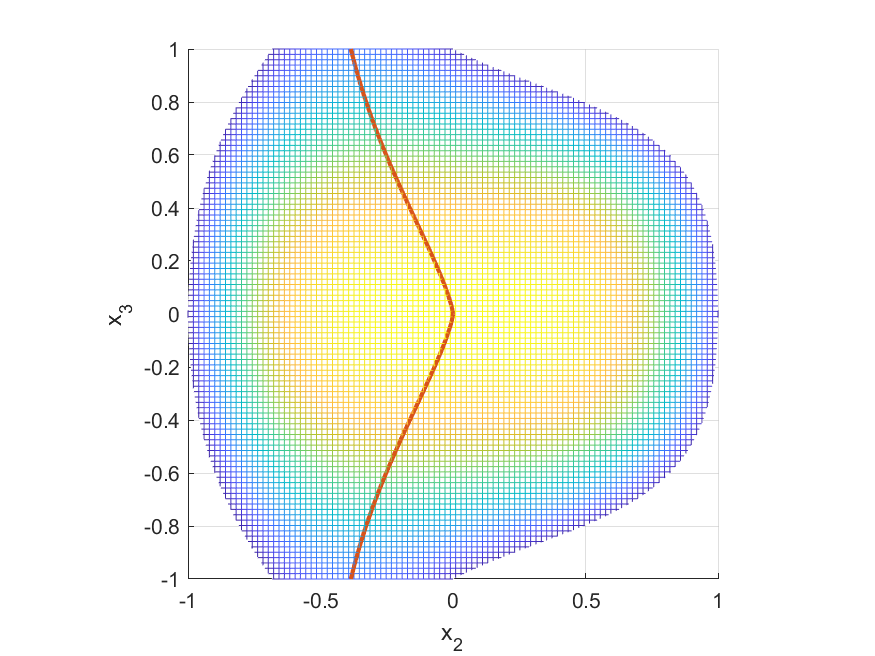} 
   \includegraphics[width=0.3\textwidth]{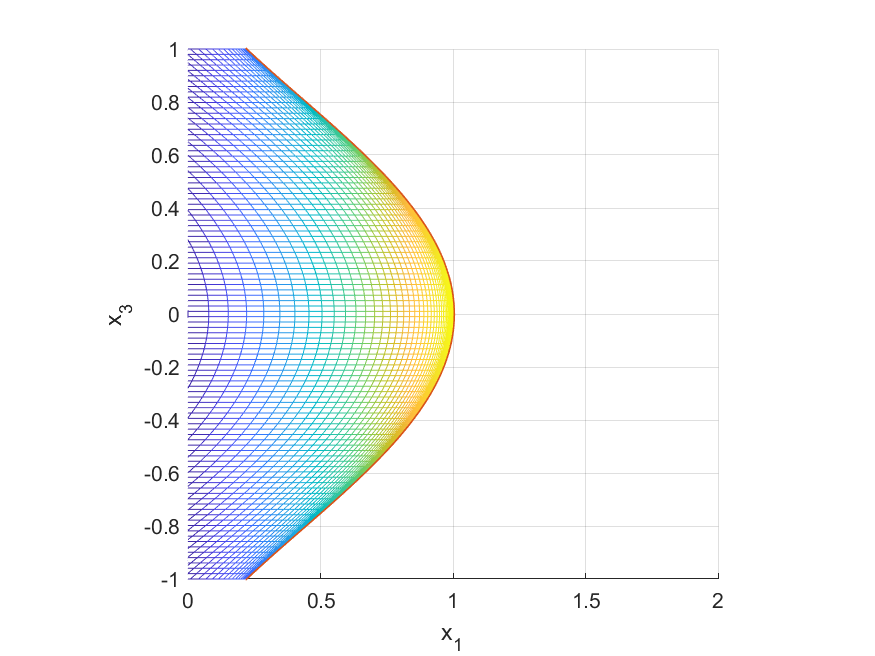}
   \caption{Light source at $(1,-\infty,0)$. First row: Cusped grazing set for the obstacle given by $x_1=1-(x_2^4+2x_2^2 x_3^2 + x_2 x_3^2+x_3^2)$. Second row: $C^1$ grazing set for the obstacle given by $x_1=1-(x_2^4+x_2 x_3^4 + x_2^2 x_3^4+x_3^2)$. Colored according to level sets $x_1=\mathrm{const.}$.}
   \label{f-3}
\end{figure}

\subsubsection{A grazing set that is $C^1$ but not twice differentiable for incoming plane waves}
Here is an example where the GS assumption holds just barely. 
Take $F(u,v)=1-G(u,v)$, where $G(u,v)=u^4+uv^4+u^2v^4+v^2$.     The incoming phase is $\phi_i=-t+(u,v)\cdot \otheta$, where $\otheta=(1,0)$.   Here again $\usigma\in G^4_d\setminus G^5$ and $\nabla^2 G(u,v)>0$ for $(u,v)\neq 0$.   
In this case the grazing curve is
\begin{align}\label{a8}
G_u=4u^3+v^4+2uv^4=0\Rightarrow 4u^3=-v^4(1+2u)\Rightarrow u=4^{-1/3}v^{4/3}+O(v^{8/3}), 
\end{align}
so $|u|\sim |v|^{4/3}$.   This is similar to \eqref{a10z}.  The same analysis shows that the grazing curve is $C^1$ but not $C^2$ near $(u,v)=0$.   See Figure \ref{f-3}.

\section{The reflected flow map for general convex incoming waves}\label{ciw}

We show in this section that the computations of section \ref{siw} can be generalized to a larger class of incoming waves.  Suppose $\mathcal O$ is a convex obstacle given by 
\[ \mathcal O:=\{ (x_1, \ox)| \ x_1< F(\ox),\; \ox\in\RR^{n-1} \} \]
where $F: \RR^{n-1}\to \RR$ is a $C^\infty$ strictly concave function.   We can arrange by translation and rotation  of $\mathcal{O}$ so that 
 $F$ has a max at $\ox=0$ with $F(0)=1$.   The tangent plane to $\partial\mathcal{O}$ at $(x_1,\ox)=(1,0)$ is then  $x_1=1$.

We consider incoming phases of the form 
\begin{align}\label{inc}
 \phi_i(x) = -t+\psi_i(x), \ (\partial_{x_1}\psi_i)^2+|\nabla \psi_i|^2=1,
 \end{align}
where $\psi_i$ is a convex function.
In what follows the notation $\nabla$ always means $\nabla_{\ox}$.



\begin{ass}\label{assp}
   Let $F$ and $\mathcal{O}$ be as just described.   Assume there exists $r>0$, an $\mathbb{R}^n-$open set $\Omega\ni (1,0)$, and a $C^\infty$ function $\psi:\Omega\to \mathbb{R}$ such that 
   \begin{align}\label{conditions}
   \begin{split}
   & (a)\; \Omega\cap\partial\mathcal{O}=\{(F(\ox),\ox):\ox\in B(0,r)\}\\
   &(b)\; |\nabla\psi_i|=1\text{ and }\psi_i\text{ is convex on }\Omega.
   \end{split}
   \end{align}

    
\end{ass}

\Remark
We can obtain a large supply of functions $\psi_i$ satisfying the conditions of Assumption \ref{assp} as follows.  Let $\mathcal C\subset \RR^n$ be any open convex set with $C^\infty$ boundary $\partial \mathcal{C}$.  It is well-known that the distance-to-the-boundary function $\psi_i(x)=-\mathrm{dist}(x,\partial{C})$ is smooth, convex, and satisfies $|\nabla \psi_i|=1$ for points of $\mathcal{C}$ sufficiently near $\partial\mathcal{C}$.  When $\partial{\mathcal C}$ is a plane, the incoming waves determined by \eqref{inc} are plane waves; when $\partial{\mathcal C}$ is a sphere, the incoming waves are spherical waves, and so on.


If the incoming bicharacteristics $\gamma_i$ hit the cotangent bundle on the space-time boundary at $(F(\ox),\ox,t, \xi^i_1(\ox), \oxi^i(\ox), -1)$, then there must hold 
\[ \xi_1^i(\ox) = \partial_{x_1}\psi_i(F(\ox),\ox), \ \oxi^i(\ox) = \nabla \psi_i( F(\ox),\ox ). \]
The grazing set and the illuminable set are
\begin{align}\label{grill2}
G_{\phi_i} = \{ (F(\ox),\ox,t)| \ \langle \oxi^i,\nabla F \rangle - \xi_1^i=0 \}, \ I_-=\{ (F(\ox),\ox,t)| \ \langle \oxi^i,\nabla F \rangle - \xi_1^i >0 \}. 
\end{align}
Fix any $T>0$ and any $r_1<r$ for $r$ as in Assumption \ref{assp}. In \eqref{grill2} and the rest of this section we take $\ox\in \overline{B(0,r_1)}$ and  $t\in [-T,T]$.

\begin{lemm}\label{mainl}
    Let $\Psi(\ox):=\psi_i(F(\ox),\ox)$.   If the incoming phase satisfies Assumption \ref{assp}, then:     \begin{enumerate}[label={\arabic*.}]
        \item The matrix $$\nabla^2\Psi-\xi^i_1\nabla^2F=\frac{\partial \oxi^i}{\partial\ox}+\nabla F\otimes \nabla \xi_1^i$$ is symmetric and positive semi-definite. 
      


        \item The incoming directions $\xi^i$ are non-focusing: for $\ox^*$, $\ox$  in the  $\ox$-projection of  $G_{\phi_i}\cup I$, the following holds 
        \[ \langle (F(\ox^*)-F(\ox), \ox^*-\ox), \xi^i(\ox^*)-\xi^i(\ox) \rangle \geq 0. \]
    \end{enumerate}
\end{lemm}
\begin{proof}
    1. We have $\nabla \Psi = \xi_1^i\nabla F + \oxi^i$. Differentiate this  to obtain the given identity.  On the other hand 
    a direct computation shows
\[\begin{split} 
\nabla^2 \Psi = & (\partial_{x_1}^2\psi_i)\nabla F\otimes \nabla F+ \nabla F\otimes \nabla \partial_{x_1}\psi_i + \nabla \partial_{x_1}\psi_i \otimes \nabla F +\nabla^2\psi_i + (\partial_{x_1}\psi_i)\nabla^2 F \\
= & \begin{pmatrix} (\nabla F)^T & I \end{pmatrix} \begin{pmatrix} 
\partial_{x_1}^2\psi_i & \nabla\partial_{x_1}\psi_i \\ (\nabla\partial_{x_1}\psi_i)^T & \nabla^2\psi_i \end{pmatrix} \begin{pmatrix} \nabla F \\ I \end{pmatrix} +(\partial_{x_1}\psi_i)\nabla^2 F.
\end{split}\]
The result follows since the first matrix on the right is symmetric and positive semi-definite.


    2.  Using the formula $\xi^i(\ox) = (\partial_{x_1}\psi_i, \nabla \psi_i)(F(\ox),\ox)$ and the convexity of $\psi_i$, we find 
    \[\begin{split}
        & \langle (F(\ox^*)-F(\ox), \ox^*-\ox), \xi^i(\ox^*)-\xi^i(\ox) \rangle\\
        & = - \langle (\partial_{x_1}\psi_i,\nabla \psi_i)(F(\ox^*),\ox^*), (F(\ox), \ox) - (F(\ox^*), \ox^*) \rangle \\
        & \quad - \langle (\partial_{x_1}\psi_i,\nabla \psi_i)(F(\ox),\ox), (F(\ox^*), \ox^*) - (F(\ox), \ox) \rangle\\
        & \geq -(\psi_i(F(\ox), \ox) - \psi_i(F(\ox^*),\ox^*)) - ( \psi_i(F(\ox^*),\ox^*) - \psi_i(F(\ox), \ox) )=0.
    \end{split}\]
    This completes the proof.
    \color{black}
\end{proof}

Let $\xi^r(\ox)$ be the reflected covector field on the boundary, then 
\[ \xi^r-\xi^i = c(1,-\nabla F), \ |\xi^r|=1. \]
From here we solve 
\begin{equation}\begin{split}\label{xir}
    \xi_1^r & = \xi_1^i-\frac{2}{1+|\nabla F|^2}( \xi_1^i-\langle 
\oxi^i,\nabla F \rangle ), \\
    \oxi^r & = \oxi^i+\frac{ 2\nabla F }{ 1+|\nabla F|^2 }( \xi_1^i - \langle \oxi^i,\nabla F \rangle ).
\end{split}\end{equation}
We obtain the following identities from \eqref{xir}:
\begin{equation}\begin{gathered}\label{e1}
\xi_1^i\nabla F +\oxi^i = \xi_1^r\nabla F+\oxi^r, \
\xi_1^r-\langle \nabla F,\oxi^r \rangle = -(\xi_1^i-\langle 
\oxi^i,\nabla F \rangle).
\end{gathered}\end{equation}
The reflected flow map is then given by 
\[ Z_r(s,\ox,t)=(F(\ox)+2s\xi_1^r(\ox), \ox+2s\oxi^r(\ox), t+2s ) \text{ for }s\in [0,\infty),\; \ox\in B(0,r),\; t\in\mathbb{R}. \]

\begin{prop}\label{mainp}
    Fix any $s_0>0$. The reflected flow map $Z_r$ is a diffeomorphism from $[0,s_0]\times I_-$ onto its range and extends to a homeomorphism from $[0,s_0]\times (G_{\phi_i}\cup I_-)$ onto its range.
\end{prop}

\begin{proof}
{\bf Injectivity.} 
It suffices to show that 
\[ z(s,\ox):=(F(\ox), \ox)+2s \xi^r(\ox) \]
is injective from the $(s,\ox)$-projection of $G_{\phi_i}\cup I_-$ to its range. Suppose the contrary, then there exists $(s,\ox)\neq (s^*,\ox^*)$ in the $(s,\ox)$-projection of $G_{\phi_i}\cup I_-$ such that\footnote{This assumption implies $\ox\neq \ox^*$.} 
\begin{equation}\label{con} 
z(s,\ox)=z(s^*,\ox^*), \text{ i.e., } (F(\ox^*)-F(\ox), \ox^*-\ox) = 2s\xi^r(\ox)-2s^*\xi^r(\ox^*). 
\end{equation}
Using the second equation in \eqref{con}  we find 
\begin{equation}\begin{split}\label{contra1}
\langle (F(\ox^*)-F(\ox), \ox^*-\ox), \xi^r(\ox^*)-\xi^r(\ox) \rangle = -2(s+s^*)(1-\langle \xi^r(\ox^*), \xi^r(\ox) \rangle)\leq 0
\end{split}\end{equation}
with equality holding if and only if
\begin{equation}\label{cond1}
    \xi^r(\ox^*) = \xi^r(\ox).
\end{equation}

On the other hand, we show that for all $\ox, \ox^*$, there holds
\begin{equation}
\label{def}
\langle (F(\ox^*)-F(\ox), \ox^*-\ox), \xi^r(\ox^*)-\xi^r(\ox) \rangle\geq 0, 
\end{equation}
with equality holding if and only if 
\begin{equation}\begin{gathered}\label{cond2} 
\xi_1^i(\ox)-\langle \oxi^i(\ox), \nabla F(\ox) \rangle = \xi_1^i(\ox^*)-\langle \oxi^i(\ox^*), \nabla F(\ox^*) \rangle=0, \text{ and }\\
\langle (F(\ox^*)-F(\ox), \ox^*-\ox), \xi^i(\ox^*)-\xi^i(\ox) \rangle=0. 
\end{gathered}\end{equation}
\color{black}
Indeed, we have
\begin{align}\label{q1}
\begin{split}
    & \langle (F(\ox^*)-F(\ox), \ox^*-\ox), \xi^r(\ox^*)-\xi^r(\ox) \rangle \\
    & \overset{\eqref{e1}}{=} (F(\ox^*)-F(\ox))(\xi_1^r(\ox^*)-\xi_1^r(\ox))+\langle \ox^*-\ox, \oxi^i(\ox^*)-\oxi^i(\ox) \rangle\\
    & \qquad + (\xi_1^r(\ox^*)-\xi_1^i(\ox^*))\langle \nabla F(\ox^*), \ox-\ox^* \rangle + (\xi_1^r(\ox)-\xi_1^i(\ox))\langle \nabla F(\ox), \ox^*-\ox \rangle \\
    & \geq  (F(\ox^*)-F(\ox))(\xi_1^r(\ox^*)-\xi_1^r(\ox))+\langle \ox^*-\ox, \oxi^i(\ox^*)-\oxi^i(\ox) \rangle\\
    & \qquad + (\xi_1^r(\ox^*)-\xi_1^i(\ox^*))( F(\ox)-F(\ox^*) ) + (\xi_1^r(\ox)-\xi_1^i(\ox))( F(\ox^*)-F(\ox) )\\
    & = \langle (F(\ox^*)-F(\ox), \ox^*-\ox), \xi^i(\ox^*)-\xi^i(\ox) \rangle \geq 0.
\end{split}
\end{align}
Here we used \eqref{xir} and the definition of $G_{\phi_i}\cup I_-$, which imply that
\begin{align}\label{q2}
\begin{split}
\xi_1^r(\ox^*)-\xi_1^i(\ox^*) & = \frac{2}{1+|\nabla F(\ox^*)|^2}\left( 
\langle \oxi^i(\ox^*), \nabla F(\ox^*) \rangle - \xi_1^i(\ox^*) \right)\geq 0,  \\
\xi_1^r(\ox)-\xi_1^i(\ox) & = \frac{2}{1+|\nabla F(\ox)|^2}\left( 
\langle \oxi^i(\ox), \nabla F(\ox) \rangle - \xi_1^i(\ox) \right)\geq 0.
\end{split}
\end{align}
Since $F$ is strictly concave, with \eqref{q2}  we conclude  that  overall equality holds in \eqref{q1}  if and only if 
\[\begin{gathered} 
\xi_1^i(\ox)-\langle \oxi^i(\ox),\nabla F(\ox) \rangle = \xi_1^i(\ox^*) - \langle \oxi^i(\ox^*), \nabla F(\ox^*) \rangle=0, \text{ and }\\ 
 \langle (F(\ox^*)-F(\ox), \ox^*-\ox), \xi^i(\ox^*)-\xi^i(\ox) \rangle=0.
\end{gathered}\]

Combining \eqref{contra1} and \eqref{def}, we conclude that \eqref{cond1} and \eqref{cond2} must be true. The first condition in \eqref{cond2}, \eqref{xir}, and \eqref{cond1} together imply that 
\begin{equation}
    \xi^i(\ox^*)=\xi^r(\ox^*) = \xi^r(\ox) = \xi^i(\ox).
\end{equation}
This and \eqref{con} give 
\begin{equation}\label{diff} 
(F(\ox^*)-F(\ox), \ox^*-\ox) = 2(s-s^*)\xi^i(\ox). 
\end{equation}
From \eqref{cond2} we have $\xi_1^i(\ox)=\langle \oxi^i(\ox), \nabla F(\ox) \rangle$.  
Multiplying this equation  by $2(s-s^*)$ and using \eqref{diff}, we obtain 
\[ F(\ox^*)-F(\ox) = \langle \nabla F(\ox), \ox^*-\ox \rangle.  \]
Since $F$ is strictly concave, this can be true only if $\ox=\ox^*$, which implies that $(s,\ox)=(s^*,\ox^*)$ by \eqref{con}. This contradicts the assumption.

The map $Z_r$ is continuous and injective on the compact domain $[0,s_0]\times(G_{\phi_i}\cup I_-)$, and hence is a homeomorphism on this domain.

\color{black}

{\bf Diffeomorphism.}
Recall that 
\[ A = I-\frac{\oxi^r\otimes \nabla F}{\xi_1^r} + 2s\left( \frac{\partial\oxi^r}{\partial\ox} - \frac{ \oxi^r\otimes \nabla \xi_1^r }{ \xi_1^r} \right). \]
Differentiating $|\xi^r|^2=1$ we find 
\begin{align}\label{q3}
 A = I-\frac{\oxi^r\otimes \nabla F}{\xi_1^r} + 2s\left( I + \frac{ \oxi^r\otimes \oxi^r }{ (\xi_1^r)^2} \right)\frac{\partial \oxi^r}{\partial\ox}. 
 \end{align}
Recalling \eqref{e1}, we differentiate 
\[ \oxi^r = \oxi^i+(\xi_1^i-\xi_1^r)\nabla F \] 
to obtain
\[ \frac{\partial \oxi^r}{\partial\ox} = \frac{\partial\oxi^i}{\partial\ox}+\nabla F\otimes \nabla (\xi_1^i-\xi_1^r)+(\xi_1^i-\xi_1^r)\nabla^2 F. \]
Differentiate the defining formula \eqref{xir} of $\xi_1^r$ to find 
\[\begin{split}
    \nabla (\xi_1^i-\xi_1^r) = & 2\frac{ \nabla \xi_1^i -\nabla F\frac{\partial \oxi^i}{\partial\ox} - \oxi^i\nabla^2 F }{ 1+|\nabla F|^2 } - 2\frac{ \oxi^r-\oxi^i }{1+|\nabla F|^2}\nabla^2 F \\
    = & -2\frac{ \oxi^i+\xi_1^i\nabla F }{\xi_1^i(1+|\nabla F|^2)} \frac{\partial\oxi^i}{\partial\ox}-\frac{2\oxi^r\nabla^2 F}{1+|\nabla F|^2}.
\end{split}\]
Therefore we find 
\[\begin{split}
    \frac{\partial\oxi^r}{\partial\ox} = & \frac{\partial\oxi^i}{\partial\ox}-\frac{ 2\nabla F\otimes (\oxi^i+\xi_1^i\nabla F) }{ \xi_1^i(1+|\nabla F|^2) }\frac{\partial\oxi^i}{\partial\ox} -\frac{2}{1+|\nabla F|^2}(\nabla F\otimes \oxi^r)\nabla^2 F + (\xi_1^i-\xi_1^r)\nabla^2 F \\
    = & \left( I-\frac{2\nabla F \otimes (\oxi^i+\xi_1^i\nabla F)}{ \xi_1^i (1+|\nabla F|^2) } \right) \frac{ \partial \oxi^i }{\partial\ox}+(\xi_1^i-\xi_1^r)\left( I-\frac{ 2\nabla F\otimes \oxi^r }{ (1+|\nabla F|^2)(\xi_1^i-\xi_1^r) } \right)\nabla^2 F.
\end{split}\]
In the following we denote 
\[\begin{split} 
B:= & I-\frac{ \oxi^r \otimes \nabla F }{\xi_1^r}, \ C:=I + \frac{ \oxi^r\otimes \oxi^r }{ (\xi_1^r)^2}, \\ 
K:= & \left( I-\frac{2\nabla F \otimes (\oxi^i+\xi_1^i\nabla F)}{ \xi_1^i (1+|\nabla F|^2) } \right) \frac{ \partial \oxi^i }{\partial\ox}, \\
L:= & (\xi_1^i-\xi_1^r)\left( I-\frac{ 2\nabla F\otimes \oxi^r }{ (1+|\nabla F|^2)(\xi_1^i-\xi_1^r) } \right)\nabla^2 F.
\end{split}\]
This allows us to rewrite $A$ as in \eqref{q3} as 
\begin{align*}\label{q5}
A=B+2sC\frac{ \partial \oxi^r }{\partial\ox}=B+2sC(K+L)=B[I+2sB^{-1}C(B^{-1})^T(B^TK+B^TL)].
\end{align*}

We compute
\[ B^TK=\left( I-\frac{\nabla F\otimes \oxi^r}{\xi_1^r} - \frac{ 2\nabla F\otimes (\oxi^i+\xi_1^i\nabla F) }{\xi_1^i(1+|\nabla F|^2)} + \frac{2 \langle \oxi^r,\nabla F \rangle \nabla F\otimes (\oxi^i+\xi_1^i\nabla F) }{\xi_1^i\xi_1^r(1+|\nabla F|^2)} \right)\frac{\partial\oxi^i}{\partial\ox}. \]
By the definition of $\oxi^r$ \eqref{xir} we have 
\[\begin{split} 
\langle \oxi^r,\nabla F \rangle 
= & \langle \oxi^i,\nabla F \rangle +|\nabla F|^2(\xi_1^i-\xi_1^r) \\
= & \frac{ 1+|\nabla F|^2 }{2}(\xi_1^r-\xi_1^i)+|\nabla F|^2(\xi_1^i-\xi_1^r)+\xi_1^i \\
= & \frac{1+|\nabla F|^2}{2}\xi_1^i+\frac{1-|\nabla F|^2}{2}\xi_1^r.
\end{split}\]
This implies\footnote{To obtain the second equality in  \eqref{q6}, use \eqref{e1} to replace $\oxi^i+\xi_1^i\nabla F$ by $\oxi^r+\xi_1^r\nabla F$.}
\begin{align}\label{q6}
\begin{split} 
B^TK = & \left(I-\frac{ \nabla F\otimes \oxi^r }{\xi_1^r} + \frac{ \xi_1^i-\xi_1^r }{\xi_1^i\xi_1^r}\nabla F\otimes (\oxi^i+\xi_1^i\nabla F)\right)\frac{\partial \oxi^i}{\partial\ox} \\
= & \left( I+\frac{ \nabla F \otimes ( (\xi_1^i-\xi_1^r)\nabla F-\oxi^r ) }{\xi_1^i} \right)\frac{\partial \oxi^i}{\partial\ox} \\
= & \left( I-\frac{\nabla F\otimes \oxi^i}{\xi_1^i} \right)\frac{\partial \oxi^i}{\partial\ox} 
=  \frac{\partial \oxi^i}{\partial\ox}+\nabla F\otimes \nabla \xi_1^i = \nabla^2 \Psi-\xi^i_1\nabla^2F.
\end{split}
\end{align}
Now apply Lemma \ref{mainl} to see that $B^TK$ is symmetric and positive semi-definite.







We now simplify $B^TL$. Notice that
\[ (B^{-1})^T = I+\frac{\nabla F\otimes \oxi^r}{\xi_1^r-\langle 
\nabla F,\oxi^r \rangle} = I-\frac{\nabla F\otimes \oxi^r}{\xi_1^i-\langle 
\nabla F,\oxi^i \rangle}. \]
On the other hand, using the definition \eqref{xir} of $\xi_1^r$, we have 
\[ L=(\xi_1^i-\xi_1^r)\left( I-\frac{\nabla F\otimes \oxi^r}{\xi_1^i-\langle \oxi^i,\nabla F \rangle} \right)\nabla^2 F=(\xi_1^i-\xi_1^r)B^{-T}\nabla^2 F. \]
Thus we have 
\begin{align}\label{btl}
B^TL = (\xi_1^i-\xi_1^r)\nabla^2 F\geq 0. 
\end{align}
Arguing as in the case of spherical waves,  we obtain 
\[\begin{split} 
    j(s,\ox,t)= & 2\xi_1^r\det A=2\xi_1^r\det(B)\det(I+2sB^{-1}C(K+L))\\
    = & 2(\langle \oxi^i,\nabla F \rangle -\xi_1^i )\det(I+2sB^{-1}C(B^{-1})^T(B^TK+B^TL)) \\
    \geq & 2(\langle \oxi^i,\nabla F \rangle -\xi_1^i )>0.
\end{split}\]
This completes the proof.
\end{proof}

\bibliographystyle{alpha}
\bibliography{GeoOptics}

\begin{thebibliography}{WW23}

\bibitem[Che96]{cheverry1996}
Christophe Cheverry.
\newblock Propagation d'oscillations pr\`es d'un point diffractif.
\newblock {\em J. Math. Pures Appl.}, 75:419--467, 1996.

\bibitem[Mel75]{melrose1975duke}
Richard~B. Melrose.
\newblock Microlocal parametrices for diffractive boundary value problems.
\newblock {\em Duke Math. J.}, 42:605--635, 1975.

\bibitem[MS78]{melrosesjostrand1978cpam}
Richard~B. Melrose and Johannes Sj\"{o}strand.
\newblock Singularities of boundary value problems. {I}.
\newblock {\em Communications on Pure and Applied Mathematics}, 31:593--617,
  1978.

\bibitem[RV73]{robertsvarberg1973}
A.~Wayne Roberts and Dale~E. Varberg.
\newblock {\em Convex Functions}.
\newblock Academic Press, 1973.

\bibitem[Wil22]{williams2022}
Mark Williams.
\newblock Solving eikonal equations by the method of characteristics.
\newblock
  \url{https://markwilliams.web.unc.edu/wp-content/uploads/sites/19674/2022/01/eikonal.pdf},
  2022.

\bibitem[WW23]{ww2023}
Jian Wang and Mark Williams.
\newblock Transport of nonlinear oscillations along rays that graze a convex
  obstacle to any order.
\newblock {\em Preprint}, 2023.

\end{thebibliography}

\end{document}